\documentclass[11pt]{article}
\usepackage{verbatim} 
\usepackage{graphicx}
\usepackage[utf8]{inputenc}
\usepackage{amsmath,amsfonts,amssymb,amscd,amsthm,txfonts,hyperref}
\usepackage{xcolor}
\newtheorem{theorem}{Theorem}[section]
\newtheorem{proposition}[theorem]{Proposition}
\newtheorem{lemma}[theorem]{Lemma}
\newtheorem{corollary}[theorem]{Corollary}
\newtheorem{remark}{Remark}[section]
\newtheorem{theo}{Theorem}

\theoremstyle{definition}
\newtheorem{definition}[theorem]{Definition}

\newcommand{\R}{\mathbb{R}}
\newcommand{\C}{\mathbb{C}}
\newcommand{\N}{\mathbb{N}}
\newcommand{\Z}{\mathbb{Z}}
\newcommand{\T}{\mathbb{T}}

\newcommand{\lap}{\bigtriangleup}

\newcommand{\E}{\mathbb E}

\newcommand{\re}{\textrm{Re}}

\newcommand{\an}[1]{\langle #1 \rangle}

\begin{document}

\title{Invariant measure for the Schr\"odinger equation on the real line}
\author{Federico Cacciafesta\footnote{SAPIENZA - Universit\'a di Roma, Dipartimento di Matematica, Piazzale A. Moro 2, I-00185 Roma, Italy}\; and Anne-Sophie de Suzzoni\footnote{Universit\'e Paris 13, Sorbonne Paris Cit\'e, LAGA, CNRS ( UMR 7539), 99, avenue Jean-Baptiste Cl\'ement, F-93430 Villetaneuse, France}}

\maketitle

\begin{abstract} In this paper, we build a Gibbs measure for the cubic defocusing Schr\"odinger equation on the real line with a decreasing interaction potential, in the sense that the non linearity $|u|^2u$ is multiplied by a function $\chi$ which we assume integrable and smooth enough. We prove that this equation is globally well-posed in the support of this measure and that the measure is invariant under the flow of the equation. What is more, the support of the measure (the set of initial data) is disjoint from $L^2$.\end{abstract}

\tableofcontents

\section{Introduction}

Our aim is to build an invariant measure for the cubic defocusing Schr\"odinger equation on the real line. 

Such issue has been dealt with by Burq-Thomann-Tzvetkov in \cite{BTTlong} when the equation presents a potential. The interest of using a well-chosen potential is that it traps the solution, in the sense that it forbids it to be too big, or even make $0$ on some domain if the potential is infinite. In \cite{BTTlong} , the equation is given by
\begin{equation}\label{potschro}
i\partial_t u + \lap u - V u  \pm |u|^2 u = 0
\end{equation}
with $V = |x|^2$. The operator $-\lap + V$ admits a countable number of eigenvalues and eigenfunctions, which is a necessary condition to build an invariant with the following method. Namely, a function on $\R$ is considered as an element $(u_n)_n$ of $\C^\N$ by writing 
$$
u = \sum_n u_n e_n
$$
with $(e_n)_n$ the eigenfunctions of $-\lap + V$. The linear or Gaussian part of the measure is then built as the measure $\mu$ induced by the random variable
$$
\varphi = \sum_n \frac{g_n}{\sqrt{\lambda_n}}e_n
$$
where $g_n$ are independent complex centred and normalised Gaussian variables. This measure can be interpreted as 
$$
d\mu \Big( \sum_n u_n e_n \Big) = \prod_{n}e^{-\lambda_n |u_n|^2}\frac{ \lambda_n du_nd\overline{u_n}}{2\pi} =  e^{-\int \overline u (-\lap + V)u} ``dL(u)"\; ,
$$
that is, a measure with density $e^{-E_c(u)}$ with regard to the Lebesgue measure, would it exist on $\C^\N$, where $E_c(u) = \int \overline u (-\lap + V)u$ is the kinetic energy of the equation \eqref{potschro}. This interpretation acquires a meaning in finite dimension.

In the defocusing case $i\partial_t u + \lap u - V u - |u|^2 u = 0$, the invariant measure is
$$
d\rho(u) = De^{-H_p(u)}d\mu(u)
$$
where $H_p(u) = \frac12 \int |u|^4$ is the potential energy and $D = \|e^{-H_p(u)}\|_{L^1_\mu}^{-1}$ is a normalisation factor. Then, $\rho$ can be interpreted as 
$$
d\rho(u) = e^{-H(u)} ``dL(u)"
$$
with $H(u) = H_c(u) + H_p(u)$ the total energy, invariant under the flow of this equation. This measure is supported in some Sobolev space $H^s$.

In this paper, we deal with the following equation
\begin{equation}\label{oureq}
i\partial_t u + \lap u - \chi |u|^2 u = 0
\end{equation}
where the potential of interaction $ \chi$ is a non negative function smooth and integrable enough (we require $0\leq \chi \lesssim \an{x}^{-\alpha}$ and $|(1-\lap)^{s_0/2} \chi | \lesssim \an{x}^{-\alpha}$ for some $\alpha > 1$ and $s_0 > 1/4$). The absence of potential prevents us to use a decomposition of $u$ in eigenfunctions of $-\lap$. Nevertheless, we can use measures as in the work by McKean-Vaninsky \cite{MVsta} , such that the Gaussian part $\mu$ of the measure is the one induced by the random variable 
$$
\varphi(x) = \int_{\R} e^{inx} \frac{dW_n(\omega)}{\sqrt{1+n^2}}
$$
where $\omega$ is the random event and $W_n$ is a Brownian motion, which makes $\varphi(x)$ a It\^o integral. Note that $1+n^2$ does not correspond to the spectrum of $-\lap$ but of $1-\lap$, which makes the interpretation of $\mu$
$$
d\mu(u) = e^{- \|u\|_{L^2}^2 - H_c(u)}``dL(u)"
$$
with $H_c (u) = \int \overline u (-\lap) u$ instead of $d\mu(u) = e^{- H_c(u)}``dL(u)"$. However, we use the invariance of the $L^2$ norm under the flow of \eqref{oureq} to prove the invariance of the final measure. The potential energy does not present any problem to build the invariance measure as $d\rho(u) = e^{-H_p(u)}d\mu(u)$. 

The main difficulty with this measure is that it is supported in $H^s_\textrm{loc}$, $s<1/2$ or in a weighted $H^s$. In the case of the Klein-Gordon equation, this problem has been solved in \cite{dSito} using the finite propagation speed of the dispersion relation of this equation, $\omega (k) = \sqrt{1+k^2}$. Here, in the Schr\"odinger equation case, we cannot use such an argument as the dispersion relation is given by $\omega(k) = k^2$ and $\frac{d\omega}{dk} = 2k$, which goes to $\infty$ when $k$ goes to $\infty$. However, we deal with this problem by applying smooth cut-offs, in the spirit of the work by Burq, G\'erard and Tzvetkov in \cite{BGTstr} . We refer the reader to Appendix A. 

Another issue is that we need a topology defined on a space containing the support of the measure $\rho$ (a weak enough topology) and for which at least the linear flow (the flow of $i\partial_t u - \lap u = 0$) is continuous. We build this topology (we call it the topology of $X^{-1/2-}$) as the one induced by the family of norms $\|\cdot\|_{\sigma_n}$ defined as
$$
\|f\|_{\sigma_n}  = \| \an{t}^{-2} \an{x}^{-2} L(t) (1-\lap)^{\sigma_n/2}  f\|_{L^2_{t,x}}
$$
with $\sigma_n  = -\frac12 -\frac{1}{n}$ for all $n \geq 1$. 

The result is the following 
\begin{theo}\label{theo-result}There exists a measure $\rho$ such that \begin{itemize}
\item the support of $\rho$ does not intersect $L^2$, that is $\rho(L^2) = 0$,
\item the equation \eqref{oureq} is globally well-posed on the support of $\rho$, we call its flow $\psi(t)$,
\item the measure $\rho$ is invariant on the topological $\sigma$-algebra of $X^{-1/2-}$, that is, for all measurable set $A$ and all time $t \in \R$, $\rho(\psi(t)^{-1} A) = \rho(A)$.
\end{itemize}
\end{theo}

The proof can be sketched in the following way. We build finite dimensional spaces $E_n$ such that the closure of $\bigcup_n E_n$ is $X^{-1/2-}$. We build measures $ \rho_n$ on $E_n$ such that $\rho_n$ converges weakly towards $\rho$. We introduce approaching equations whose flows are written $\psi_n(t)$ such that $E_n$ is stable under $\psi_n(t)$ and $\rho_n$ is invariant under $\psi_n(t)$. We prove that $\psi_n(t)$ is globally well-defined on the support of $\rho$. We prove that \eqref{oureq} is locally well-posed, and that $\psi_n(t)$ converges towards $\psi(t)$ locally in time. Then, we extend the local properties to global times. We deduce the invariance of the measure $\rho$ under $\psi(t)$ using these convergences.

For the local analysis, we have to prove local well-posedness on periodic settings. Because of that, we heavily rely on Strichartz estimates on compact manifolds of the work of Burq-Gérard-Tzvetkov \cite{BGTstr}. It provides a Strichartz estimate as on $\R^d$ but with loss of derivative.

To sum up, we stress on the facts that the main difference with \cite{BTTlong} is the lack of a trapping potential, and the main difference with \cite{dSito} is the absence of finite speed propagation. 

What is more, we mention the paper by Bourgain, \cite{Binvinf}, which proves the weak convergence of Schr\"odinger equation on a periodic setting of period $L$ towards a solution of the Schr\"odinger equation on $\R$ when $L$ goes to $\infty$, provided that the initial data are taken in the support of specific measures. This problem is close to ours.

\paragraph{Organisation of the paper} The next subsection sums up the notations and the assumption we make on the potential of interaction $\chi$.

In Section 2, we define $\rho_n$, $\rho$, prove that $(\rho_n)_n$ converges weakly towards $\rho$, and discuss their supports. We also deal with the linear case $\chi = 0$, and prove finite dimensional invariance.

In Section 3, we prove all the local in time properties, like local well-posedness and local convergence of $\psi_n(t)$ towards $\psi(t)$, along with continuity of the flows. In this section, we do not use the measures.

In Section 4, we extend the local properties to global times on a certain set of initial data $A$. 

In Section 5, we prove that $A$ is of full measure and that the measure $\rho$ is invariant under the flow of \eqref{oureq}.

\paragraph{Acknowledgements} The first author is supported by the FIRB 2012 "Dispersive dynamics, Fourier analysis and variational methods".

\subsection{Notations}

For the rest of the paper, we need a certain amount of notations that we fix here.

First of all, we need two real numbers $s_0 \in ]1/4, 1/2[$ and $s_\infty \in ]1/4, s_0[$. The regularity of the solutions will be some $s$ taken between $s_\infty$ and $s_0$.

We also need a real number $\alpha \in ]1, \frac32[$.

We fix some $p \in ]4, \frac1{1/2 - s_\infty}[$. By $\gamma'$, we denote $1 - \frac3p$ and by $\gamma $, we denote $\frac1{\gamma'}$. The numbers $\alpha$ and $p$ are needed in the definition of the norms used in the control of the initial data.

The norm $\|.\|_{X_*,Y_\circ}$ means
$$
\|u\|_{X_*,Y_\circ} = \| * \mapsto \|\circ \mapsto u(x,\circ) \|_Y \|_X
$$
where $*$ and $\circ$ may be replaced by time variables ($t$, $\tau$), space variables ($x$), or random events ($\omega$). The space $X_*,Y_\circ$ is the space of functions normed by $\|.\|_{X_*,Y_\circ}$.

We set $D$ the operator $\sqrt{1-\lap}$. What is more, we note $\an x = \sqrt{1+x^2}$.

We work on the following spaces, with $t_0 \in \R$, $T\leq 1$ and $s\in [s_\infty, s_0]$ and $R \in \R_+ \cup \{+ \infty\}$ : 
$$
Y_{t_0,T} (s,R) = \mathcal C^\infty ([t_0-T,t_0+T], H^s(R)) \cap L^p([t_0-T,t_0+T], L^\infty (R))
$$
with $\|\cdot\|_{H^s(R)} = \|D^s . \|_{L^2([-R,R])}$ and $L^\infty(R) = L^\infty ([-R,R])$.

By definition $Y_T(s,R) = Y_{0,T}(s,R)$ and $Y_T(s) = Y_T (s,+ \infty)$. In particular, we will prove the local well-posedness of our equations in $Y_T(s, \cdot)$.

For all $s\leq s_0$, and all $t_0 \in \R$, we write $Z_{t_0}(s)$ the space corresponding to the norm
$$
\|\cdot\|_{Z_{t_0} (s)} = \|\an{x}^{- \alpha}D^s \cdot \|_{L^p([t_0-1,t_0+1], L^2(\R))}+ \|\an{x}^{-\alpha}\cdot\|_{L^p([t_0-1,t_0+1], L^\infty (\R))}
$$
and by definition $Z(s) = Z_0 (s)$.

We also introduce the space $Z'(s)$ associated to the norm
$$
\|\cdot \|_{Z'(s)} = \|\an{t}^{- 2}\an{x}^{- \alpha}D^s \cdot \|_{L^p_t,L^2_x}+ \|\an{t}^{- 2}\an{x}^{-\alpha}\cdot\|_{L^p_t, L^\infty_x}\; .
$$

The number $p$ has been chosen such that we have the following Strichartz estimate on the torus $ N_k \T$, with $f$ a $2\pi N_k$ periodic function belonging to $H^s$ :
$$
\|L(t) f \|_{Y_T(s, \pi N_k)} \leq C \|f\|_{H^s(\pi N_k)}
$$
with a constant $C$ independent from $T \leq 1$, $k$ (because Strichartz estimates are scale invariant) and $s$ (because it is chosen in a compact $[s_\infty, s_0]$). Thanks to a $TT^*$ argument, 
$$
\|\int_{0}^t L(t-\tau) f(\tau) d\tau\|_{Y_T(s,\pi N_k)} \leq \int_{0}^t \|f(\tau)\|_{H^s(\pi N_k)} \; .
$$

We also introduce the Fourier multiplier $\Pi_k$ such that
$$
\widehat{\Pi_k f}(n) = \eta \Big(\frac{n}{M_k}\Big) \hat f (n)
$$
with $\eta$ a non negative even $\mathcal C^\infty $ function with compact support included in $[-1,1]$ and such that $\eta = 1$ on $[-1/2,1/2]$ and $M_k$ a sequence going to $\infty$. In the sequel, $M_k$ will be $k$. We write $\eta_k(n) = \eta \Big(\frac{n}{M_k}\Big)$.

Thanks to the smoothness of $\eta_k$ , we have
$$
\|\Pi_k f\|_{L^p} \leq C_p \|f\|_{L^p}
$$
for all $p$ and all $k$, including $p=\infty$. The proof is essentially contained in \cite{BGTstr}, but in this particular case, a direct proof results from the smoothness of $\eta$ and consequent estimates on its Fourier transform.

The operator $\Pi_k$ is a smooth cutoff for the high frequencies. It allows us to consider the linear Schr\"odinger flow as one with finite propagation speed in the following sense. Let $1_R$ be function equal to $1$ on $[-R,R]$ and $0$ otherwise, we have for $T \geq |t|$,
$$
\|1_R L(t) \Pi_k f\|_{L^p}  \lesssim \|1_R L(t) \Pi_k ( 1_{R+3TM_k} f)\|_{L^p} + \frac1{TM_k^2}\sup_y \|f\|_{L^p([y-R,y+R])}\; .
$$
Indeed, the Kernel of $L(t) \Pi_k$ is given by 
$$
K_t (z) = \int_{\R} e^{-in^2 t} \eta (\frac{n}{M_k}) e^{inz}dn
$$
and if $|z| > 3T M_k$, then by  a double integration by parts, we get
$$
|K_t(z) | \leq C_\eta \frac1{M_k} \frac1{z^2}
$$
and hence its $L^1$ norm where $z$ is restricted as $|z|\geq 3M_k |t|$ is less than $C_\eta \frac1{TM_k^2}$. This proof is in Appendix A.

Finally, the map $\psi_k(t)$ is the flow of the equation
\begin{equation}\label{approach}
i\partial_t u + \lap u - \Pi_k P_k (\chi |\Pi_k u|^2 \Pi_k u) = 0 \; ,
\end{equation}
the map $\psi(t)$, the flow of 
\begin{equation}\label{NLS}
i\partial_t u + \lap u - \chi |u|^2 u = 0 \; ,
\end{equation}
and 
\begin{equation}\label{psiprime}
\psi_k'(t) = \psi_k(t) - L(t)\; ,\; \psi'(t) = \psi(t) - L(t) \; .
\end{equation}

\paragraph{Assumptions on $\chi$}

We assume that for the $\alpha \in ]1,\frac32[$ and for the $s_0 \in ]\frac14,\frac12[$ defined above, there exists $C$ such that for all $x$,
$$
\chi^{1/3}(x) \leq C \an{x}^{-\alpha} \mbox{ and } |D^{s_0}\chi^{1/3}(x)| \leq C \an{x}^{-\alpha} \; .
$$

\section{Definitions and properties of measures}

In this section, we introduce the Gaussian measure along with its approximation on finite dimension. We gives properties of these measures and their supports. We prove the invariance of the Gaussian measure under the linear flow. Then, we define the invariant measure $\rho$ and a sequence of approaching measures $(\rho_k)_k$ on finite dimension. We prove that this sequence converges weakly towards $\rho$ and that $\rho_k$ is invariant under the flow $\psi_k(t)$.

\subsection{Definition and approximation of the linear measure}

In this subsection, we define the random variable $\varphi$ needed to build the measure $\mu$ invariant through the linear flow and precise into which spaces it belongs.

The random variable $\varphi$ is defined as the limit of a sequence of random variables. Let us describe this sequence.

For the rest of this paper, we call $(\Omega, \mathcal F, \mathbb P)$ a probability space and $(W_n)_{n\in \R}$ the union of two complex Brownian motions with the same initial value.

The sequence, and random variable $\varphi$ is the same as in \cite{dSito}.

\begin{definition}\label{def-phi} Let $N,M \in \N^*$. We call $\varphi_{N,M}$ the random variable defined as : 
$$
\varphi_{N,M} (\omega, x) = \sum_{k=-NM}^{NM-1} \delta_{N,k} (\omega) \frac1{\sqrt{1+\frac{k^2}{N^2}}}e^{ikx/N}
$$
where $\omega \in \Omega$ is an event of the probability space, $x\in \R$ is the space variable and $\delta_{N,k} = W_{\frac{k+1}{N}}-W_{\frac{k}{N}}$.
\end{definition}

\begin{remark}\label{rem-Glaw} This random variable is a Gaussian vector. Indeed, $\varphi_{N,M}$ is entirely determined by $2NM$ Gaussian variables $a_{-NM},\hdots, a_{NM-1}$ with $a_k= \delta_{k,N} (1+\frac{k^2}{n^2})^{-1/2}$. The law of $a_k$ is $\mathcal N(0, \frac1{N(1+ k^2/N^2)})$ and the $a_k$ are independent from each other. Hence they form a Gaussian vector whose law is given by the covariance matrix $M(N)$ such that
$$
M(N)_{i,j} = \E (\overline a_i a_j) = \beta_i^j \frac1{N(1+\frac{j^2}{N^2})}\; 
$$
where $\beta_i^j = 1$ if $i=j$ and $0$ otherwise.

Its law is given by 
$$
(\mbox{det } M(N))^{-1/2}e^{-\langle a,M(N)^{-1} a\rangle}\prod_{k= -NM}^{NM-1} \frac{da_kd\overline a_k}{2\pi} 
$$
where
$$
\langle a,M(N)^{-1}a \rangle = \sum_{k= -NR}^{NR-1} |a_k|^2 N \left( 1+\frac{k^2}{N^2}\right)
$$
can be rewritten as 
$$
\frac1{2\pi}\int_{-\pi N}^{\pi N} \overline v(x) (1-\Delta)v(x)dx
$$
with $v$ given by
$$
v(x) = \sum_{k=-MN}^{MN-1} a_k e^{ikx/N}\; .
$$
\end{remark}

\begin{proposition}Let $s<-1/2$. Let $L(t)$ be the flow of the linear Schr\"odinger equation, that is $i\partial_t u = -\lap u $. The sequence $D^s L(t)\varphi_{2^n, M}$ converges in $\an t \an x L^\infty_t, L^\infty_x, L^2_\omega$ when $n$ goes to $\infty$, uniformly in $M$. We call its limit $D^sL(t)\varphi_M$. In other words, for all $\varepsilon >0$, there exists $n_0\in \N$ such that for all $M$ and all $n\geq n_0$, 
$$
\|(\an t \an x)^{-1}\Big( D^s L(t) \varphi_M -  D^sL(t)\varphi_{2^n,M}\Big)\|_{L^\infty_t,L^\infty_x,L^2_\Omega}\leq \varepsilon \; .
$$\end{proposition}

\begin{proof} We follow the same argument as in \cite{dSito} .

We take $n \geq m$. We have for all $x,t,\omega, M$ 
$$
D^s L(t) \varphi_{2^n,M}(x, \omega) - D^sL(t) \varphi_{2^m,M}(x,\omega) = \sum_{l=-2^m M}^{2^mM-1} \sum_{j=0}^{2^{n-m}-1} \delta_{2^n,2^{n-m}l+j}(\omega) \Big( f_{t,x,s}\Big(\frac{2^{n-m}l+j}{2^n}\Big) - f_{t,x,s}\Big(\frac{l}{2^m}\Big)\Big)
$$
with 
$$
f_{t,x,s}(n) = \frac{e^{in^2 t}e^{inx}}{(1+n^2)^{(1-s)/2}}  \; .
$$
This is due to the fact that 
$$
\delta_{2^m, l} = W_{\frac{l+1}{2^m}} -W_{\frac{l}{2^m}} = \sum_{j=0}^{2^{n-m}-1} W_{\frac{l}{2^m}+\frac{j+1}{2^n}}-W_{\frac{l}{2^m}+\frac{j}{2^n}} \; .
$$

The derivative of $f_{t,x,s}$ is given by
$$
f'_{t,x,s}(n) = e^{in^2t}e^{inx}\Big( \frac{2int + ix}{(1+n^2)^{(1-s)/2}} + \frac{n(s-1)}{(1+n^2)^{(3-s)/2}}\Big)
$$
from which we deduce the bound
$$
|f'_{x,t,s}(n)|\leq |s-1|(\an t \an x) (1+n^2)^{s/2}  \; .
$$
By taking the $L^2_\omega$ norm to the square of the functions we compare, we get
$$
\|D^s L(t) \varphi_{2^n,M}(x, \omega) - D^sL(t) \varphi_{2^m,M}(x,\omega)\|_{L^2_\omega}^2 \leq 2 |s-1|^2 (\an t \an x)^{2} \sum_{l=-2^mM}^{2^m M- 1} \Big(1+\frac{l^2}{2^{2m}}\Big)^{s} \sum_{j=0}^{ 2^{n-m}-1} \frac1{2^n}\frac{j^2}{2^{2n}}\; .
$$
This is due to the independence of the $\delta_{2^n, 2^{n-m}l +j}$.

We remark that
$$
 \sum_{j=0}^{ 2^{n-m}-1} \frac1{2^n}\frac{j^2}{2^{2n}}\leq 2^{-3m}
$$
thus, since 
$$
\frac1{2^m} \sum_{l= -2^m M}^{2^m M -1} \Big( 1+ \frac{l^2}{2^{2m}}\Big)^s\leq 2 \int_{\R} (1+y^2)^s dy\; ,
$$
we get
$$
\|D^s L(t) \varphi_{2^n,M}(x, \omega) - D^sL(t) \varphi_{2^m,M}(x,\omega)\|_{L^2_\omega}^2 \leq 4 \cdot 2^{-2m} |s-1|^2 (\an t \an x)^{2} \int (1+y^2)^{s} dy\; .
$$
The integral converges since $s< -1/2$. Finally by dividing by $\an t \an x$ and taking the $L^\infty_t,L^\infty_x$ norm, we get 
$$
\|(\an t \an x)^{-1}\Big(D^s L(t) \varphi_{2^n,M}(x, \omega) - D^sL(t) \varphi_{2^m,M}(x,\omega)\Big)\|_{L^\infty_t,L^\infty_x,L^2_\Omega}\leq C_s 2^{-m}
$$
which concludes the proof.
\end{proof}

\begin{proposition} The sequence $(D^sL(t)\varphi_M)_M$ converges in $\an t \an x\an t \an xL^\infty_t,L^\infty_x, L^2_\omega$. We call its limit $D^sL(t)\varphi$. \end{proposition}

\begin{proof}We refer to \cite{dSito}. \end{proof}

\begin{definition}We call $\phi_k = \varphi_{2^k,k}$. The sequence $(\phi_k)_k$ converges towards $\varphi$ in the norm \\
$\|(\an x \an t))^{-1} D^s L(t)\cdot \|_{L^\infty_t,L^\infty_x,L^2_\Omega}$ for $s<-1/2$.\end{definition}

\begin{remark} In the rest of the paper, we call $M_k = k$ and $N_k = 2^k$. We use these notations only when $M_k$ refers to the cutoff in frequency and $N_k$ the period of some functions. \end{remark}

\subsection{Properties of \texorpdfstring{$\mu$}{mu}}

We precise the spaces to which $\varphi$ belongs.

\begin{proposition}\label{prop-belong} For all $1\leq p < \infty$ and $s<1/2$ and $t\in \R$, $L(t)\varphi$ belongs to $ L^p_\omega,W^{s,p}_{\textrm{loc},x}$ and for all $\xi \in L^p_x$, $\xi L(t)\varphi$ belongs to $L^p_\omega, L^p_x$. We also have that for all $1\leq q < \infty$, $\xi L(\tau) \varphi$ belongs to $L^q_{\textrm{loc},\tau }(\R, L^p_x)$ which we write as $L^q_{\textrm{loc},\tau }, L^p_x$ and $L(\tau)\varphi$ belongs to $L^q_{\textrm{loc},\tau},L^p_{\textrm{loc},x}$.  \end{proposition}

\begin{proposition}For $\mathbb P$-almost every $\omega \in \Omega$, $\varphi (\omega)$ does not belong to $L^2_x$. \end{proposition}

\begin{proof}We refer to \cite{dSito} for the proofs of those two propositions. Indeed, the proofs only rely on the fact that $L(t)$ is a Fourier multiplier $\widehat{L(t) f}(n) = \alpha_t(n) \widehat f (n)$ with $|\alpha_t(n)|=1$.\end{proof}

\begin{remark} We have that for all $k$ the $L^r_\omega, L^p_x$ of $G \phi_k$ is bounded independently from $k$ for all $G$ in $L^p$. This is how we get Proposition \ref{prop-belong}. \end{remark}

\begin{remark} As $\varphi$ and $\phi_k$ are Gaussian variables, we have that if $N$ is a norm on functions and if $N(\varphi)$ (resp. $N(\phi_k)$) is $\mathbb P$-almost surely finite then there exist $C$ and $a$ depending continuously in $\|N(\varphi)\|_{L^2_\omega}$ (resp. $\|N(\phi_k)\|_{L^2_\omega}$) such that 
$$
\mathbb P( N(\varphi)\geq \Lambda) \leq C e^{-a \Lambda^2} \; , \; \mathbb P( N(\phi_k)\geq \Lambda) \leq C e^{-a \Lambda^2} \; .
$$
In particular, with $s\in [s_\infty, s_0]$ and $N = \|\an{x}^{-\alpha}D^s \cdot\|_{L^2}$, $a$ is bounded from below independently from $s$ and $C$ is bounded independently from $s$ too. This is a part of Fernique's theorem, \cite{Fint} .\end{remark}

\begin{definition} We call $\mu$ the measure induced by $\varphi$ and $\mu_k$ the one induced by $\phi_k$. The $\sigma$-algebras on which $\mu$ is defined are the same as the one of the spaces to which $\varphi$ belongs. \end{definition}

\begin{remark}The sequence $\mu_k$ converges weakly towards $\mu$ in the topological $\sigma$- algebra of $\|\an{ t}^{-1} \an{ x}^{-1} D^s L(t) \cdot \|_{L^2_{t,x}} $. But $\mu$ is defined for bigger $\sigma$ algebras as the one of $\|\an{x}^{-\alpha} D^{s_0} \cdot\|_{L^2}$ or $\|\an{x}^{-\alpha}\cdot \|_{L^\infty}$. \end{remark}

\begin{proposition}\label{prop-infty} Let $p\geq 1$, $\alpha > 1$ and $\beta > 1/p$. The random variable $\langle x \rangle^{-\alpha} \langle t \rangle^{-\beta} L(t) \varphi $ belongs almost surely to $L^p_t,L^\infty_x$.\end{proposition}

\begin{proof} We divide $\R$ as the union of the $[n,n+1]$ to get
\begin{eqnarray}\label{somme}
\|\langle x \rangle^{-\alpha} L(t) \varphi\|_{L^\infty (\R)}  & \leq & \sum_{n\in\Z} \|\langle x \rangle^{-\alpha} L(t) \varphi\|_{L^\infty([n,n+1])} \\
 & \lesssim & \sum_{n\in\Z} \langle n \rangle^{-\alpha}\| L(t) \varphi\|_{L^\infty([n,n+1])} 
\end{eqnarray}

Let $\xi$ a $\mathcal C^\infty$ function such that $\xi = 1$ on $[0,1]$ and $\xi = 0$ on the complementary set of $]-1,2[$. Let $\xi_n (x) = \xi (x-n)$, we have 
$$
\|\langle x \rangle^{-\alpha} L(t) \varphi\|_{L^\infty (\R)}   \lesssim  \sum_{n\in\Z} \langle n \rangle^{-\alpha}\| \xi_n L(t) \varphi\|_{L^\infty([n-1,n+2])} \; .
$$
We use Sobolev embedding with $s< 1/2$ and $q > 1/s$ to get
\begin{eqnarray*}
\| \xi_n L(t) \varphi\|_{L^\infty([n-1,n+2])} &\lesssim & \| \xi_n L(t) \varphi\|_{W^{s,q}([n-1,n+2])} \\
 & \lesssim & \| D^s\xi_n L(t) \varphi\|_{L^q([n-1,n+2])} \; .
\end{eqnarray*}
Taking the $L^r_\omega$ norm (in probability) of this quantity with $r=\max (q,p)$, we get thanks to \\ Minkowski inequality
$$
\| \xi_n L(t) \varphi\|_{L^r_\omega,L^\infty([n-1,n+2])} \leq \|D^s \xi_n L(t) \varphi\|_{L^q([n-1,n+2]),L^r_\omega}\; .
$$
By definition of $\varphi$, we have that $L(t) \varphi$ is the It\^o integral
$$
L(t)\varphi = \int \langle m \rangle^{-1} e^{i m^2 t }e^{imx } dW_m
$$
hence, by multiplying it by $\xi_n$ and applying $D^s$ we get
$$
D^s \xi_n L(t) \varphi = \int \langle m \rangle^{-1} e^{i m^2 t }D^s (\xi_n e^{imx }) dW_m
$$
and its $L^r_\omega$ norm is bounded by
$$
\| D^s \xi_n L(t) \varphi\|_{L^r_\omega} \leq C \sqrt r \Big( \int \langle m \rangle^{-2}| D^s (\xi_n e^{imx })|^2 dm \Big)^{1/2}\; .
$$
Taking its $L^q$ norm gives 
$$
\| D^s \xi_n L(t) \varphi\|_{L^q,L^r_\omega} \leq C \sqrt r \Big( \int \langle m \rangle^{-2}\| D^s (\xi_n e^{imx })\|_{L^q}^2 dm \Big)^{1/2}\; .
$$

In order to evaluate $\| D^s (\xi_n e^{imx })\|_{L^q}$, we consider $ \xi_n e^{imx }$ as a $3$-periodic function. We write 
$$
\xi_n (x) e^{imx} = \sum_{k \in \Z} \alpha_k e^{i2\pi k x /3} 3^{-1/2}
$$
with 
$$
\alpha_k = \int_{n-1}^{n+1} \xi_n (x) e^{imx} e^{-i2\pi kx/3} dx \; .
$$
Let $k_0(m)$ be the unique $k\in \Z$ such that $|m-2\pi k/3| < \frac \pi 3$. We have 
$$
|\alpha_{k_0(m)} | \leq 3\|\xi\|_{L^\infty}
$$
and for $k\neq k_0(m)$, thanks to a double integration by parts,
$$
|\alpha_k| \leq \frac{3\|\xi"\|_{L^\infty}}{(m-2\pi k/3)^2} \leq \frac{6  \|\xi"\|_{L^\infty}}{\langle m-2\pi k/3\rangle^2}\; .
$$
We write $D^s \xi_n e^{imx}$ as
$$
\sum_k \langle 2\pi k/3 \rangle^s \alpha_k e^{2i \pi kx/3}
$$
and take its $L^q$ norm to get
$$
\| D^s (\xi_n e^{imx })\|_{L^q} \leq \sum_k \langle 2\pi k/3 \rangle^s |\alpha_k| 3^{1/p}\; .
$$
By inputting the estimate on $\alpha_k$, we get
$$
\| D^s (\xi_n e^{imx })\|_{L^q} \leq C (\langle \frac{2\pi k_0(m)}{3}\rangle^s + \sum_k \langle \frac{2\pi k}{3}\rangle^s (\langle \frac{2\pi k}{3}-m\rangle^{-2})
$$
with a constant $C$ depending on $\xi$. Since by definition of $k_0(m)$, we have $\langle \frac{2\pi k_0(m)}{3}\rangle^s \leq C_s \langle m\rangle^s$, and because $s\geq 0$, we have  $\langle\frac{2\pi k}{3}\rangle^s\leq C_s \langle \frac{2\pi k}{3}-m\rangle^s  + \langle m\rangle^s$, we have 
$$
\| D^s (\xi_n e^{imx })\|_{L^q} \leq C \langle m \rangle^s 
$$
with $C$ depending on $s$ and $\xi$.

Finally, we get
$$
\|D^s \xi_n L(t) \varphi\|_{L^q([n-1,n+2], L^r_\omega)} \leq C \Big(  \int \langle m \rangle^{2s-2} dm \Big)^{-1/2}\; .
$$
with $C$ depending on $s$ and $\xi$ and the integral converges as $s<1/2$. By inputting this estimate in \eqref{somme} after taking the $L^r_\omega$ norm, we get
$$
\|\langle x\rangle^{-\alpha} L(t) \varphi\|_{L^r_\omega,L^\infty} \lesssim \sum_n \langle n \rangle^{-\alpha} 
$$
and the sum converges since $\alpha > 1$. We have, as $r\geq p$,
\begin{eqnarray*}
\|\langle t \rangle^{-\beta} \langle x\rangle^{-\alpha} L(t) \varphi\|_{L^r_\omega, L^p_t, L^\infty_x} &\leq & \|\langle t \rangle^{-\beta} \langle x\rangle^{-\alpha} L(t) \varphi\|_{L^p_t, L^r_\omega, L^\infty_x}\\
 & \lesssim & \|\langle t \rangle^{-\beta}\|_{L^p}
\end{eqnarray*}
and this norm is finite since $\beta > 1/p$. Therefore, the $L^r_\omega, L^p_t, L^\infty_x$ of $\langle t \rangle^{-\beta} \langle x\rangle^{-\alpha} L(t) \varphi$ is finite which means that the $L^p_t, L^\infty_x$ norm of $\langle t \rangle^{-\beta} \langle x\rangle^{-\alpha} L(t) \varphi$ is almost surely finite and concludes the proof. \end{proof}

\begin{corollary} For all $p \geq 1$, all $\alpha > 1$ and all $\beta > 1/p$, there exists $C, c > 0$ such that for all $\Lambda$
$$
\mathbb P \Big( \|\langle t \rangle^{-\beta} \langle x\rangle^{-\alpha} L(t) \varphi\|_{L^p_t,L^\infty_x} > \Lambda \Big) \leq C e^{-c\Lambda^2}\; .
$$
\end{corollary}

\begin{proof}The proof is a consequence of Fernique's theorem, \cite{Fint}. \end{proof}

\subsection{Invariance of \texorpdfstring{$\mu$}{mu} under the linear flow}

We introduce the family of norms
\begin{equation}\label{semin}
p_{s}(f):f\rightarrow \left\|\langle t\rangle^{-2}\langle x\rangle^{-2}D^sL(t)f(x)\right\|_{L^2_tL^2_x},\qquad s\in\mathcal{A}:=\left\{-\frac12-\frac1l:l\in\mathbb{N}^*\right\}.
\end{equation}

Since this family is countable ($\mathcal{A}$ is countable), the metric $d$ given by
\begin{equation}\label{distance}
d(u,v)=\sum_{l\in\mathbb{N}^*}2^{-l}\frac{p_{1/2-1/l}(u-v)}{1+p_{1/2-1/l}(u-v)}
\end{equation}
is equivalent to the topology induced by the norms $p_{q,s}$. We call $X^{-1/2-}$ the functional space equipped with that metric.

\paragraph{Approximation of $\mu$ and invariance under the linear flow}

We here mean to show the invariance of the measures $\mu_k$ and $\mu$ under the linear flow.

\begin{definition}
We denote with $\mu^t$ (resp. $\mu_k^t$) the image measures of $\mu$ (resp. $\mu_k)$ under the flow $L(t)$, that is for all measurable set $A\subset X^{-1/2-}$ 
\begin{equation*}
\mu^t(A)=\mu(L(t)^{-1}A)
\end{equation*}
\begin{equation*}
\mu_k^t(A)=\mu_k(L(t)^{-1}A).
\end{equation*}
\end{definition}

We now prove the invariance of these measures under the linear flow.
\begin{theorem}\label{invariance}
The measures $\mu_k$ and $\mu$ are invariant under the flow $L(t)$.
\end{theorem}
\begin{proof}
We begin with the approximating measures $\mu_k$. The random variable $\phi_k$ can be seen as a Gaussian vector $(x_1,...,x_{N_k})$ where $x_1,...,x_{N_k}$ are independent and $\mathbb{E}(x_i)=0$ (see Remark 0.1). Applying $L(t)$ to it yields
\begin{equation*}
\left(e^{i\alpha_1t}x_1,...,e^{i\alpha_{N_k}t}x_{N_k}\right)
\end{equation*}
with $\alpha_i\in\mathbb{R}$. The Gaussian variable $e^{i\alpha_jt}x_j$ has the same law of $x_j$ and is independent from all the other $x_i$, hence the law of $\phi_k$ is the same as the law of $L(t)\phi_k$. In other words, the measure $\mu_k$ is invariant under $L(t)$.

We now turn to $\mu$: we therefore need to show that for every measurable set $A\subset X^{-1/2-}$,  $\mu(L(t)^{-1}A)=\mu(A)$. We need two preliminary results: first of all we show the following approximating property.

\begin{lemma}\label{applem}
For all open set $U\subset X^{-1/2-}$ and closed set $F\subset X^{-1/2-}$ we have
\begin{equation}\label{appr}
\mu(U)\leq \liminf_{k\rightarrow+\infty}\mu_k(U),\qquad
\mu(F)\geq \limsup_{k\rightarrow+\infty}\mu_k(F).
\end{equation}
\end{lemma}
\begin{proof}
This consists in the weak convergence of the measures $\mu_k$ towards $\mu$. To prove the weak convergence, it is enough to show that
\begin{equation*}
\mathbb{E}(p_{1/2-1/l}(\varphi-\phi_k))\underset{k\rightarrow\infty}\longrightarrow0.
\end{equation*}
In fact we have
\begin{eqnarray*}
\mathbb{E}(p_{1/2-1/l}(\varphi-\phi_k))&\leq&\|p_{1/2-1/l}(\varphi-\phi_k)\|_{L^2(\Omega)}
\\
&\leq&
\left\|\langle\tau\rangle^{-2}\langle x\rangle^{-2}D^sL(\tau)(\varphi-\phi_k)\right\|_{L^2_\tau L^2_xL^2(\Omega)}
\\
&\leq&
\|D^sL(\tau)(\varphi-\phi_k)\|_{L^\infty_\tau L^\infty_xL^2(\Omega)}.
\end{eqnarray*}
Since $D^s L(t)\phi_k$ converges towards $D^sL(t)\varphi$ in $\langle x\rangle \langle t\rangle L^\infty_tL^\infty_xL^2_\omega$ and thus in $\langle x\rangle^2 \langle t\rangle^2 L^\infty_tL^\infty_xL^2_\omega$, the proof is concluded.
\end{proof}

Then we need the continuity of the linear flow on $X^{-1/2-}$.
\begin{lemma}\label{contl}
The flow $L(t)$ is continuous on $X^{1/2-}$.
\end{lemma}
\begin{proof}
We prove that for every for every $s\in\mathcal{A}$ there exists a constant $C=C(t)$ such that
\begin{equation}\label{targ}
p_s(L(t)f)\leq C(t)p_s(f).
\end{equation}
By definition \eqref{semin} and the change of variable $\tau\rightarrow\tau-t$  we have, recalling that $\langle\tau\rangle\leq\langle t\rangle\langle\tau-t\rangle$,
\begin{eqnarray*}
p_s(L(t)f)&=&\left(\int\int\langle\tau\rangle^{-4}\langle x\rangle^{-4}|D^sL(t+\tau)f(x)|^2d\tau dx\right)^{1/2}
\\
&=&
\left(\int\int\langle\tau-t\rangle^{-4}\langle x\rangle^{-4}|D^sL(\tau)f(x)|^2d\tau dx\right)^{1/2}
\\
&\leq&
\left(\int\int\left(\frac{\langle t\rangle}{\langle\tau\rangle}\right)^{4}\langle x\rangle^{-4}|D^sL(\tau)f(x)|^2d\tau dx\right)^{1/2}
\\
&\leq&
C(t)p_s(f)
\end{eqnarray*}
and the proof is concluded.

\end{proof}
We now return to the proof of Theorem \ref{invariance}; the argument will now follow closely some previous works such as \cite{Binvkdv,Zoni}, anyway, we include it here for the seek of completeness.

Let $K\subset X^{-1/2-}$ be a closed set, and
\begin{equation*}
K_\epsilon:=\{u\in X^{-1/2-}:\exists v\in K: d(u,v)<\epsilon\}
\end{equation*}
where the distance $d$ is given in \eqref{distance}. Since $K_\epsilon\subset X^{-1/2-}$ is open and $L(t)$ is continuous in $X^{-1/2-}$ (see Lemma \ref{contl}), the set $L(t)^{-1}K_\epsilon$ is open, too. Lemma \ref{applem} then implies
\begin{equation*}
\mu^t(K)\leq \mu^t(K_\epsilon)\leq \liminf_{k\rightarrow\infty}\mu^t_k(K_\epsilon).
\end{equation*}
Using the first part of the Theorem (invariance of $\mu_k$) then yields
\begin{equation*}
\mu^t(K)\leq\liminf_{k\rightarrow\infty}\mu_k(K_\epsilon)\leq\limsup_{k\rightarrow+\infty}\mu_k(\overline{K_\epsilon}),
\end{equation*}
where $\overline{K_\epsilon}$ is the closure of $K_\epsilon$. Lemma \ref{applem} implies again ($\overline{K_\epsilon}$ is closed)
\begin{equation*}
\mu^t(K)\leq\mu(\overline{K_\epsilon})
\end{equation*}
which in turns implies
\begin{equation*}
\mu^t(K)\leq \inf_{\epsilon>0}\mu(\overline{K_\epsilon}).
\end{equation*}
The dominated convergence theorem for $\epsilon\rightarrow 0$ then implies
\begin{equation*}
\mu^t(K)\leq\mu(K).
\end{equation*}
Since now $L(t)$ is continuous on $X^{-1/2-}$, we use reversibility of the flow to conclude that $L(t)^{-1}(K)$ is closed and hence
\begin{equation*}
\mu(K)=\mu^{-t}(L(t)^{-1}K)\leq \mu(L(t)^{-1}K)=\mu^t(K).
\end{equation*}
For all closed sets $K$ we have therefore proved that $\mu^t(K)=\mu(K)$; since closed sets in $X^{-1/2-}$ generate its topological $\sigma$-algebra, the proof is concluded.
\end{proof}

\subsection{Definition and approximation of the non linear measure}

In this subsection, we define the invariant measure $\rho$, the approaching sequence $(\rho_k)_k$ and prove that the sequence converges weakly towards $\rho$.

\begin{definition} We set 
$$
f_k(u) = e^{-\frac1{2\pi}\int_{-\pi N_k}^{\pi N_k}\chi |\Pi_k u|^4} \; , \; f(u) = e^{-\frac1{2\pi}\int_{\R}\chi |u|^4} 
$$
and we define the measures $\rho$ and $\rho_k$ as 
$$
d\rho_k(u) = \frac1{J_k} f_k(u) d\mu_k(u) \; , \; d\rho(u) = \frac1J f(u) d\mu(u)
$$
with $J_k = \|f_k\|_{L^1_{\mu_k}}$ and $J = \|f\|_{L^1_\mu}$.
\end{definition}

The rest of this subsection is dedicated to the proof of the weak convergence of $\rho_k$ towards $\rho$ in the topological $\sigma$-algebra of $X^{-1/2-}$.

\begin{lemma} \label{lem-cvm1}For all $k$ and all $j$, we have 
$$
\|\chi^{1/2} (1-\Pi_k) \varphi \|_{L^2_\omega,L^2_x} \lesssim \|\chi\|_{L^1}^{1/2} \frac1{\sqrt{M_k}} \; , \; \|\chi^{1/2} (1-\Pi_k) \phi_j \|_{L^2_\omega,L^2_x} \lesssim \|\chi\|_{L^1}^{1/2} \frac1{\sqrt{M_k}}\; .
$$
\end{lemma}

\begin{proof} For $\varphi$, we have 
$$
\|\chi^{1/2} (1-\Pi_k) \varphi \|_{L^2_\omega,L^2_x} \leq \|\chi\|_{L^1}^{1/2} \|(1-\Pi_k) \varphi\|_{L^\infty_x, L^2_\omega} \; .
$$
As $(1-\Pi_k) \varphi$ is given by
$$
(1-\Pi_k)\varphi(x) = \int (1-\eta_k) (n) \frac{e^{inx}}{\sqrt{1+n^2}} dW_n(\omega)
$$
we have 
$$
\|(1-\Pi_k)\varphi(x)\|_{L^2_\omega}^2 = \int |(1-\eta_k) (n)|^2 \frac{1}{1+n^2} dn 
$$
and given the definition of $\eta_k$
$$
\|(1-\Pi_k)\varphi(x)\|_{L^2_\omega}^2 \lesssim \frac1{M_k} \; .
$$

For $\phi_j$, we use the same argument with a sum instead of an integral
$$
(1-\Pi_k) \phi_j (x) = \sum_{l = -N_j M_j}^{N_j M_j -1 } (1-\eta_k(j/N_j)) \frac{e^{ijx/N_j}}{\sqrt{1+(j/N_j)^2}} \delta_{N_j, l} 
$$
hence 
$$
\|(1-\Pi_k) \phi_j (x)\|_{L^2_\omega} = \sum_{l = -N_j M_j}^{N_j M_j -1 } |(1-\eta_k(j/N_j))|^2 \frac{1}{1+(j/N_j)^2} \frac1{N_j} \lesssim \frac1{M_k}\; .
$$\end{proof}

\begin{lemma}\label{lem-cvm2} The sequence $\|\chi^{1/2} (\varphi - \phi_k)\|_{L^2_\omega, L^2_x}$ goes to $0$ when $k$ goes to $\infty$. \end{lemma}

\begin{remark} The loss of regularity in the convergence in $d$ is due to the evolution in time. \end{remark}

\begin{proof}Due to \cite{dSito}, we know that $\phi_k$ converges to $\varphi$ in $\an{x} L^\infty_x , L^2$. We have 
$$
\|\chi^{1/2} (\varphi - \phi_k)\|_{L^2_\omega, L^2_x} \leq \|\an x \chi^{1/2}\|_{L^2_x} \|\an{x}^{-1} (\varphi - \phi_k)\|_{L^\infty, L^2_\omega}\; .
$$
As $\chi$ is less than $\an{x}^{-3 \alpha}$ with $\alpha > 1$, we have that $\an x \chi^{1/2}$ is less than $\an{x}^{- (3\alpha/2 -1)}$, and since $3\alpha/2 -1 > 1/2$, $\an x \chi^{1/2}$ belongs to $L^2$.
\end{proof}

\begin{lemma} \label{lem-cvm3} The sequences $\|f-f_k\|_{L^1_{\mu_k}} $ and $\|f-f_k\|_{L^1_\mu}$ go to $0$ when $k$ goes to $\infty$. \end{lemma}

\begin{proof}We use that
$$
\Big| f(u) - f_k(u) \Big| \leq \int_{|x|\geq \pi N_k} \chi |u|^4 + \|\chi^{1/2} (u- \Pi_k u)\|_{L^2_x} (\|\chi^{1/2} |u|^3\|_{L^2_x} + \|\chi^{1/2} |\Pi_k u|^3\|_{L^2_x})\; .
$$
For the first term, we use H\"older inequality,
$$
\int_{|x|\geq \pi N_k} \chi |u|^4  \leq \|\chi^{1/2}\|_{L^2(|x|\geq \pi N_k)} \|\chi^{1/2} |u|^4\|_{L^2_x}
$$
and the fact that 
$$
\|\chi^{1/2} |u|^4\|_{L^1_{\mu_k},L^2_x}
$$
 is finite and bounded uniformly in $k$. The sequence $ \|\chi^{1/2}\|_{L^2(|x|\geq \pi N_k)}$ goes to $0$ when $k$ goes to $\infty$.

For the second one, we use that 
$$ 
(\|\chi^{1/2} |u|^3\|_{L^2_{\mu_k},L^2_x} + \|\chi^{1/2} |\Pi_k u|^3\|_{L^2_{\mu_k},L^2_x})
$$
is finite and uniformly bounded in $k$, that
$$ 
(\|\chi^{1/2} |u|^3\|_{L^2_{\mu},L^2_x} + \|\chi^{1/2} |\Pi_k u|^3\|_{L^2_{\mu},L^2_x})
$$
satisfies the same properties and that
$$
\|\chi^{1/2} (u- \Pi_k u)\|_{L^2_{\mu_k},L^2_x } \; , \; \|\chi^{1/2} (u- \Pi_k u)\|_{L^2_{\mu},L^2_x }
$$
converge towards $0$ thanks to Lemma \ref{lem-cvm1}.
\end{proof}

\begin{lemma}\label{lem-cvm4} The sequence $f\circ \phi_k$ converges towards $f\circ \varphi$ in $L^1_\omega$. \end{lemma}

\begin{proof} This is due to the fact that 
$$
|f(u) - f(v)| \leq \|\chi^{1/2} (u-v)\|_{L^2_x} (\|\chi^{1/2}|u|^3\|_{L^2_x} + \|\chi^{1/2}|v|^3\|_{L^2_x})\; ,
$$
that $\|\chi^{1/2}|\varphi|^3\|_{L^2_\omega,L^2_x} + \|\chi^{1/2}|\phi_k|^3\|_{L^2_\omega,L^2_x}$ is uniformly bounded  in $k$ and Lemma \ref{lem-cvm2}.\end{proof}

\begin{proposition}\label{prop-cvm} The sequence $\rho_k$ converges weakly towards $\rho$. \end{proposition}

\begin{proof} Let $F$ be a bounded Lipschitz continuous function on $X^{-1/2-}$. We have 
$$
\Big| \int F(u)  d \rho(u)  - \int F(u) d\rho_k(u) \Big| \leq I + II +III
$$
 with 
\begin{eqnarray*}
I &=& \Big| \int \frac{f(\varphi)}{J} F(\varphi) d\mathbb P -  \int \frac{f(\varphi)}{J} F(\phi_k) d\mathbb P\Big| \\
II &=& \Big| \int \frac{f(\varphi)}{J} F(\phi_k) d\mathbb P -  \int \frac{f(\phi_k)}{J} F(\phi_k) d\mathbb P\Big| \\
III &=& \Big| \int \frac{f(\phi_k)}{J} F(\phi_k) d\mathbb P -  \int \frac{f_k(\phi_k)}{J_k} F(\phi_k) d\mathbb P\Big|  \; .
\end{eqnarray*}

For $I$, we use that $f$ is bounded, $F$ is Lipschitz-continuous and that $\E_{\mathbb P}(d(\varphi, \phi_k))$ converges towards $0$ as was proved in Subsection 2.1 since $\|\cdot\|_{L^1_\omega} \leq \|\cdot\|_{L^2_\omega}$.

For $II$, we use that $F$ is bounded and Lemma \ref{lem-cvm4}.

 For $III$, we use Lemmas \ref{lem-cvm1} \ref{lem-cvm2} to prove that $J_k \rightarrow J$, that $F$ is bounded and that $\|f-f_k\|_{L^1_{\mu_k}} \rightarrow 0$.
\end{proof}

\subsection{Invariance of \texorpdfstring{$\rho_k$}{rhok} under the finite dimensional non linear flow}

In this subsection, we prove that the flow $\psi_k(t)$ is Hamiltonian on $E_k$ and that $\rho_k$ is invariant under $\psi_k(t)$.

\begin{lemma}\label{lem-hamil} The flow $\psi_k(t)$ is a a flow of a Hamiltonian equation on $E_k$ with Hamiltonian
$$
H_k = \frac1{2\pi} \Big( -\int_{-\pi N_k}^{\pi N_k} \overline u \lap u  + \frac12 \int_{-\pi N_k}^{\pi N_k} \chi |\Pi_k u|^4 \Big) \; .
$$
\end{lemma}

\begin{proof} Let 
$$
H_c(u) = -\frac1{2\pi} \int_{-\pi N_k}^{\pi N_k} \overline u \lap u = N_k\sum_j |u_j|^2 \frac{j^2}{N_k^2} \; .
$$
We have 
$$
\frac{dH_c(u)}{du_j}  = N_k \frac{j^2}{N_k^2} \overline{u_j} \; .
$$

Let 
$$
H_p (u) = \frac1{4\pi } \int (P_k \chi) |\Pi_k u|^4 = \frac12 \sum_{j_1 + \hdots + j_5 = 0} (P_k \chi)_{j_1} (\overline{\Pi_k u})_{j_2} (\overline{\Pi_k u})_{j_3} (\Pi_k u)_{j_4} (\Pi_k u)_{j_5} \; .
$$
We have 
\begin{eqnarray*}
\frac{dH_p(u)}{du_j} &=& N_k \sum_{j_1 + \hdots + j_4 = -j}  (P_k \chi)_{j_1} (\overline{\Pi_k u})_{j_2} (\overline{\Pi_k u})_{j_3} (\Pi_k u)_{j_4}\\
 &=& N_k \Big( P_k\chi |\Pi_k u|^2 \overline{\Pi_k u}\Big)_{-j} \\
 &= & N_k \overline{(P_k \chi |\Pi_k u|^2 \Pi_k u)_j} \; .
\end{eqnarray*}
By summing these two identities, we get
$$
\frac{dH_k(u)}{du_j} = N_k \overline{(i\partial_t u)_j} = -i N_k \dot{\overline{u_j}}
$$
which gives 
$$
\partial_t u = J \bigtriangledown_u H_k(u) 
$$
with $J = \frac{i}{N_k}$.
\end{proof}

\begin{lemma}The $L^2$ norm is invariant under $\psi_k(t)$. \end{lemma}

\begin{proof}The proof follows from the differentiation of 
$$
\frac12 \int |u|^2
$$
which gives 
$$
\partial_t \frac12 \int |u|^2 = \re\Big( -i \int \overline u (\lap u - (P_k \chi) |\Pi_k u|^2 \Pi_k u) \Big) = 0\; .
$$
\end{proof}

\begin{proposition}\label{prop-fininvar} The measure $\rho_k$ is invariant under the flow $\psi_k(t)$. \end{proposition}

\begin{proof} The measure $\rho_k$ is such that
$$
d\rho_k(u) =d_k e^{-\frac1{2\pi}\int |u|^2 - H_k(u)} dL_k(u)\; .
$$
The Lebesgue measure $L_k$ is invariant under the flows of Hamiltonian equations, and the $L^2$ norm and $H_k$ are invariant under $\psi_k(t)$. Hence $\rho_k$ is invariant under $\psi_k(t)$.\end{proof}

\section{Local analysis}

In this section, we prove all the local in time properties of the flows $\psi(t)$ and $\psi_k(t)$ that we extend in the next section to prove the invariance of the measure. We prove local well-posedness, local uniform convergence of $\psi_k(t)$ towards $\psi(t)$, and local continuity.

\subsection{Local well-posedness}

In this subsection, we prove the local well-posedness of \eqref{NLS} and \eqref{approach} with a time of existence independent from the regularity $s \in [s_\infty,s_0]$ in which we solve the Cauchy problem and independent from $k$, the index of the approaching sequence.

\begin{proposition}\label{prop-lwp} Let $t_0 \in \R$. There exists $C$ such that for all $s \in [s_\infty, s_0]$, all $\Lambda \geq 1$, all $k \in \N$, all $u_0$ such that
$$
\|L(t) u_0\|_{Z_{t_0}(s)} \leq \Lambda
$$
and all $v_0\in E_k $ such that $\|v_0\|_{H^s(\pi N_k)}\leq \Lambda$, the Cauchy problem 
$$
\left \lbrace{ \begin{tabular}{ll}
$i \partial_t v + \lap v - \Pi_k P_k \Big( \chi |\Pi_k (L(t) u_0 + v)|^2(\Pi_k(L(t) u_0 + v)) \Big)= 0 $\\
$v_{|t=t_0} = v_0$ \end{tabular}} \right.
$$
has a unique solution in $Y_{t_0,T}(s, \pi N_k)\cap E_k$ with $T =  \frac1{C \Lambda^{2\gamma}} $ and 
$$
\|v\|_{Y_{t_0,T}(s,\pi N_k)} \leq C \Lambda \; .
$$
\end{proposition}

\begin{proof}Following the proof of \cite{BGTstr}, we rely on Strichartz estimates to prove that the map 
$$
A(v)(t) = L(t-t_0) v_0 +i  \int_{t_0}^t L(t-\tau) \Pi_k P_k \Big( \chi |\Pi_k(L(\tau) u_0 + v(\tau))|^2(\Pi_k(L(\tau) u_0 + v(\tau))) \Big)
$$
is contracting over some closed set. The solution $v$ is the fixed point of $A$. 

Thanks to Strichartz estimates, we have 
$$
\|A(v)\|_{Y_{t_0,T}(s,\pi N_k)} \leq C\|v_0\|_{H^s(\pi N_k)} + \int_{t_0}^t \big\| \Pi_k P_k \Big( \chi |\Pi_k(L(\tau) u_0 + v(\tau))|^2(\Pi_k(L(\tau) u_0 + v(\tau))) \Big) \big\|_{H^s(\pi N_k)}\; .
$$
We have that
\begin{multline*}
\big\| \Pi_k P_k \Big( \chi |L(\tau) u_0 + v(\tau)|^2(L(\tau) u_0 + v(\tau)) \Big) \big\|_{H^s(\pi N_k)} \leq \\
\|\chi |\Pi_k L(\tau) u_0|^2\Pi_k L(t) u_0\|_{H^s(\R)} + \|(P_k \chi) |\Pi_k v|^2 \Pi_k v\|_{H^s(\pi N_k)}.
\end{multline*}
We have by definition $\chi \in L^\infty \cap H^s$ and 
$$
\||\Pi_k f|^2 \Pi_k f\|_{H^s} \leq C (\|\Pi_k f\|_{H^s}+\|\Pi_k f\|_{L^\infty})^3\; .
$$
Besides, $\Pi_k$ is a smooth cutoff, hence $\|\Pi_k f\|_{L^\infty} \leq C \|f\|_{L^\infty}$. What is more, by definition of $\chi$, $\chi^{1/3} \lesssim \an{x}^{-\alpha}$ and $|D^{s_0}\chi^{1/3}| \lesssim \an{x}^{-\alpha}$, hence 
$$
\int_{t_0}^t \|\chi |\Pi_k L(\tau) u_0|^2\Pi_k L(t) u_0\|_{H^s(\R)}d\tau \leq T^{\gamma'} \|L(t) u_0\|_{Z_{t_0}(s)}\; .
$$
Finally, we get
$$
\|A(v)\|_{Y_{t_0,T}(s,\pi N_k)} \leq C\|v_0\|_{H^s(\pi N_k)} + C T^{\gamma'} \Big( \|L(t) u_0\|_{Z_{t_0}(s)}^3 + \|\chi\|_{L^\infty \cap H^s} \|v\|_{Y_{t_0,T}(s, \pi N_k)}^3\big)
$$
and with the hypothesis on $\chi$, $u_0$ and $v_0$
$$
\|A(v)\|_{Y_{t_0,T}(s,\pi N_k)} \leq C\Lambda + C T^{\gamma'} \Big(\Lambda^3 +  \|v\|_{Y_{t_0,T}(s, \pi N_k)}^3\Big)\; .
$$
Therefore, for $T = \frac1{C \Lambda^{2\gamma}}$, the ball of radius $C\Lambda$ is stable under $A$. 

For the same reasons,
$$
\|A(v)- A(w)\|_{Y_{t_0,T}(s,\pi N_k)} \leq CT^{\gamma'} \Lambda^2 \|v-w\|_{T_{t_0,T}(s,\pi N_k)}
$$
which makes it contracting for $T = \frac1{C\Lambda^{2\gamma}}$ with $C$ big enough. Therefore, we have the existence, uniqueness of the solution as well as the bound on it.\end{proof}

\begin{remark} This involves in particular that the solutions $\psi_k(t) u_0$ exist and are unique. Indeed, $\psi'_k(t) u_0 = \psi(t) u_0 - L(t) u_0$ is the solution of the equation of \ref{prop-lwp} when $t_0 = 0$ and $v_0 = 0$. Besides, the proposition is still true if we solve 
$$
\left \lbrace{ \begin{tabular}{ll}
$i \partial_t v + \lap v -  \Big( \chi |L(t) u_0 + v|^2(L(t) u_0 + v) \Big)= 0 $\\
$v_{|t=t_0} = v_0$ \end{tabular}} \right.
$$
in $Y_{t_0, T}(s, + \infty)$ as the Strichartz are even more allowing in terms of regularity in non-compact manifolds. Hence, the solution $\psi(t)u_0$ is locally well-defined.
\end{remark}

\subsection{Local uniform convergence}

In this subsection, we prove that $\psi_k(t)$ converges uniformly in the initial datum, locally in time for several metrics.

\begin{definition}Let $Y'_T(s)$ be the space associated to the norm
$$
\|\cdot\|_{Y'_T(s)} = \|\an{x}^{-\alpha}D^{s} \cdot\|_{L^\infty_t([-T,T],L^2_x)}+\|\an{x}^{-\alpha} \cdot\|_{L^p_t([-T,T],L^\infty_x)}\; .
$$
It is a weighted version of $Y_T(s)$.
\end{definition}

\begin{lemma}\label{lem-jenaigravemarre}Let $u_0$ such that
$$
\| L(\tau) u_0\|_{Z(s)} \leq \Lambda \; .
$$

We have for all $R$, and all $s'< s$,
$$
\|\psi(t)u_0 - \psi_k(t) u_0\|_{Y'_T(s')} \leq C_\alpha(R) + C \|\psi(t)u_0 - \psi_k(t) u_0\|_{Y_T(s' , R)}
$$
with $C_\alpha(R)  = C ( \langle R \rangle^{-\alpha} + \langle R \rangle^{1-\alpha} (\pi N_k)^{-1})\rightarrow 0$ when $R$ goes to $\infty$.
\end{lemma}

\begin{proof} We focus on 
$$
\|1_{|x|\geq R} (\psi_k(t) u_0 - \psi(t) u_0)\|_{Y'_T(s')}\; .
$$
We bound it by
$$
\|1_{|x|\geq R} \psi_k(t) u_0 \|_{Y'_T(s')} + \|1_{|x|\geq R}  \psi(t) u_0\|_{Y'_T(s')} \; .
$$
We divide $\{|x|\geq R\}$ in a union of intervals $\cup J_m (R) \cup J'_m(R)$ with $J_m(R) = [R+2\pi m N_k, R +2\pi (m+1) N_k]$ and $J'_m(R) = [-R - 2\pi (m+1) N_k, R +2\pi m N_k]$. We get
$$
\|1_{|x|\geq R} \psi_k(t) u_0 \|_{Y'_T(s')} \leq C \sum_m \langle R + 2 m \pi N_l\rangle^{-\alpha} 2\pi N_k \|\psi_k(t) u_0 \|_{Y_T(s', \pi N_k)} \leq C(R) \Lambda 
$$
with $C(R) = C  ( \langle R \rangle^{-\alpha} + \langle R \rangle^{1-\alpha} (\pi N_k)^{-1})$.

For the second one, we have 
$$
\|1_{|x|\geq R} \psi(t) u_0\|_{Y'_T(s')}\leq C\langle R \rangle^{-\alpha} \Lambda \; .
$$
\end{proof}

\begin{proposition}\label{prop-luc}There exists $C$ such that for all $s\in [s_\infty, s_0]$, all $s' < s$, all $\Lambda \geq 1$, all $\varepsilon > 0$,  all $u_0$ such that
$$
\|L(t)u_0\|_{Z(s)} \leq \Lambda
$$
and all $k$,
$$
\| (\psi(t)u_0 - \psi_k(t)u_0)\|_{Y'_T(s')} \leq  C M_k^{s'-s}T^{-1} \Lambda\; .
$$
\end{proposition}

\begin{proof} Because $\psi(t)$ is the solution of \eqref{NLS}, we have 
$$
\psi(t) u_0  = L(t) u_0 + i\int_{0}^t L(t-\tau) \chi |\psi(\tau)u_0|^2\psi(\tau) u_0 d\tau\; .
$$
Similarly, we have 
$$
\psi_k(t) u_0  = L(t) u_0 + i\int_{0}^t L(t-\tau)\Pi_k P_k (\chi | \Pi_k \psi_k(\tau)u_0|^2\Pi_k \psi_k(\tau) u_0 )d\tau\; .
$$
Therefore, we can write
$$
\psi(t) u_0 - \psi_k(t) u_0  = i\int_{0}^t L(t-\tau) \Big( \chi |\psi(\tau)u_0|^2\psi(\tau) u_0  - \Pi_k P_k (\chi |\Pi_k \psi_k(\tau)u_0|^2\Pi_k \psi_k(\tau) u_0 )\Big) d\tau \; .
$$
We set
$$
F = \Big( \chi |\psi(\tau)u_0|^2\psi(\tau) u_0  - \Pi_k P_k (\chi |\Pi_k \psi_k(\tau)u_0|^2\Pi_k \psi_k(\tau) u_0 )\Big) \; .
$$
We divide $F$ into three parts $F = F_1+ F_2+F_3$ where 
\begin{eqnarray*}
F_1 &=& (1-\Pi_k) \chi|\psi(\tau) u_0|^2 \psi(\tau)u_0 \; , \\
F_2 &=& \Pi_k (1- P_k) \chi|\psi(\tau) u_0|^2 \psi(\tau) u_0\; , \\
F_3 &=& \Pi_k P_k \chi \Big( |\psi(\tau) u_0|^2 \psi(\tau) u_0 - |\Pi_k \psi_k(\tau) u_0|^2 \Pi_k \psi_k(\tau)u_0\Big) \; .
\end{eqnarray*}

Let $R_k = \pi N_k (1-\frac1k)$.

For $i=1,2,3$, let 
$$
I.i = \|\int_{0}^t L(t-\tau) F_i (\tau) d\tau\|_{Y_T(s',R_k)}
$$
such that $\|\psi(t) u_0 - \psi_k(t) u_0\|_{Y_T(s',R_k)} \leq I.1 +I.2+I.3$. 

We estimate $I.1$. We have 
$$
I.1 \leq \|\int_{0}^t L(t-\tau) F_1 (\tau) d\tau\|_{Y_T(s')} \leq \int_{0}^t \|F_1 (\tau) \|_{H^{s'}(+\infty)} d \tau
$$
As $\Pi_k$ is a smooth cutoff over the frequency $M_k$, we have 
$$
I.1 \leq \frac{M_k}{2}^{s'-s} \int_{0}^t \|\chi|\psi(\tau) u_0|^2 \psi(\tau)u_0\|_{H^{s}} \leq C M_k^{s'-s} T^{\gamma'}\|1_{t\in [-T,T]}\chi^{1/3} \psi(\tau) u_0\|_{Z(s)}^3\; .
$$
We have that $\psi(t)u_0 = L(t)u_0 + \psi'(t)u_0$ with 
$$
\|\chi^{1/3}L(t)u_0\|_{Z(s)} \leq \Lambda
$$
by hypothesis and 
$$
\|1_{[-T,T]}\chi^{1/3} \psi'(t)u_0\|_{Z(s)} \leq C_\chi \|\psi'(t)u_0\|_{Y_T(s)}\lesssim \Lambda
$$
thanks to local well-posedness. Therefore,
$$
I.1 \leq C T^{\gamma'} \Lambda^3 M_k^{s'-s}\; .
$$

For $I.2$, we use the finite propagation speed. We have $I.2 \leq A+B$ with
$$
A = \| \int_{0}^t L(t-\tau) \Pi_k 1_{R_k + 2TM_k}(1-P_k) \chi|\psi(\tau) u_0|^2 \psi(\tau) u_0\|_{Y_T(s')}\; ,
$$
and 
$$
B = \frac1{T M_k} \int_{0}^t \sup_y \| D^{s'} (1-P_k) |\psi(\tau) u_0|^2 \psi(\tau) u_0\|_{L^2([y-R_k, y+R_k])}\; .
$$
We choose $T$ such that $3T M_k  \leq \frac{\pi N_k}{k}$, that is $T \leq \frac{ \pi N_k}{3k M_k}$, we can choose $T \leq C = \min_k  \frac{ \pi N_k}{3k M_k} $. This makes $A = 0$.

For $B$, we have 
$$
\| D^{s'}  |\psi(\tau) u_0|^2 \psi(\tau) u_0\|_{L^2([y-R_k, y+R_k])} \leq \|  |\psi(\tau) u_0|^2 \psi(\tau) u_0 \|_{H^{s'}(\R)}
$$
and
$$
\| D^{s'} P_k  |\psi(\tau) u_0|^2 \psi(\tau) u_0\|_{L^2([y-R_k, y+R_k])} \leq \|  |\psi(\tau) u_0|^2 \psi(\tau) u_0 \|_{H^{s'}(\pi N_k)}
$$
as $R_k$ is less than $\pi N_k$ and $P_k f$ is $2\pi N_k$ periodic for all $f$. Finally,
$$
B\leq \frac{T^{\gamma' - 1}}{M_k^2} \Lambda^3\; .
$$
 
For $I.3$, we set $q'$ such that $\frac12 -\frac1{q'} = s-s'$ and $q = \frac1{1 - \frac1{q'}}$. We have 
\begin{multline*}
 \|\chi \Big( |\psi(\tau) u_0|^2 \psi(\tau) u_0 - |\Pi_k \psi_k(\tau) u_0 |^2 \Pi_k \psi_k(\tau) u_0\|_{H^{s'}(\pi N_k)}   \\
\leq  \|\chi \Big( |\psi(\tau) u_0|^2 \psi(\tau) u_0 - |\Pi_k \psi(\tau) u_0 |^2 \Pi_k \psi(\tau) u_0\|_{H^{s'}(\pi N_k)} + \\
\|\chi \Big( |\Pi_k \psi(\tau) u_0|^2 \Pi_k \psi(\tau) u_0 - |\Pi_k \psi_k(\tau) u_0 |^2 \Pi_k \psi_k(\tau) u_0\|_{H^{s'}(\pi N_k)} \; .
\end{multline*}
We have 
\begin{multline*}
\|\chi \Big( |\psi(\tau) u_0|^2 \psi(\tau) u_0 - |\Pi_k \psi(\tau) u_0 |^2 \Pi_k \psi(\tau) u_0\|_{H^{s'}(\pi N_k)} \leq \|\chi^{1/3} (1-\Pi_k) \psi(\tau)u_0\|_{H^{s'}}\|\chi^{2/3}(|P_k \psi(\tau) u_0|^2 +\\
|\psi(\tau) u_0|^2)|\|_{L^\infty} + \|\chi^{1/3} (1-\Pi_k) \psi(\tau)u_0\|_{L^q}\|\chi^{2/3}(|P_k \psi(\tau) u_0|^2 + |\psi(\tau) u_0|^2)|\|_{W^{s',q'}}\; .
\end{multline*}
 Since $H^s$ is embedded in $W^{s',q'}$ we have 
$$
\|\chi^{2/3}|P_k \psi(\tau) u_0|^2 + |\psi(\tau) u_0|^2|\|_{W^{s',q'}} \lesssim (\|\psi(\tau) u_0\|_{H^s} + \|\psi(\tau) u_0\|_{L^\infty})^2\; .
$$
We also have 
$$
 \|\chi^{1/3} (1-\Pi_k) \psi(\tau)u_0\|_{H^{s'}} \leq C_\chi M_k^{s-s'} \|\an{x}^{-\alpha} D^s \psi(\tau) u_0\|_{L^2}
$$
and 
\begin{multline*}
 \|\chi^{1/3} (1-\Pi_k) \psi(\tau)u_0\|_{L^q} \leq  \|\chi^{1/3} (1-\Pi_k) \psi(\tau)u_0\|_{L^2}^\theta  \|\chi^{1/3} (1-\Pi_k) \psi(\tau)u_0\|_{L^\infty}^{1-\theta}\\
 \leq C_\chi M_k^{-s \theta} \|\psi(\tau)u_0\|_{H^s}^\theta \|\psi(\tau) u_0\|_{L^\infty}^{1-\theta}
\end{multline*}
with $\theta = \frac2{p}$ which makes $s\theta \geq s-s'$. Finally, we get
$$
I.3 \leq C T^{\gamma'} \Lambda^2 \|\chi^{1/3} \psi(t)u_0 - \psi_k(t) u_0\|_{Y_T(s', \pi N_k)} + C T^{\gamma' }M_k^{s'-s} \Lambda^3
$$
$$
I.3 \leq C T^{\gamma'} \Lambda^2 \| \psi(t)u_0 - \psi_k(t) u_0\|_{Y'_T(s')}\; .
$$

In the end, we have 
$$
\| \psi(t)u_0 - \psi_k(t) u_0\|_{Y'_T(s')} \leq C(R_k) + C M_k^{s'-s}
$$
and $C(R_k) = C ( \langle R \rangle^{-\alpha} + \langle R \rangle^{1-\alpha} (\pi N_k)^{-1}) \leq C (\pi N_k)^{-\alpha} \leq C M_k^{s'-s}$ as $N_k = 2^k$ and $M_k = k$. \end{proof}

\paragraph{Further convergences}

We need further convergence to be able to extend the local properties to global times and to prove the invariance of the measure.

\begin{proposition}\label{prop-xmoinscv}Assume that $u_0$ is such that 
$$
\| L(t)u_0\|_{Z(s)} \leq \Lambda \; .
$$
Then, for all $t \in [-T,T]$ with $T = \frac1{C \Lambda^{2\gamma}}$, the sequence $\psi_k(t)u_0$ converges towards $\psi(t) u_0$ in $X^{-1/2-}$ with a rate of convergence independent from $u_0$, namely
$$
d(\psi_k(t), \psi(t)) \leq C M_k^{-\beta} \Lambda \frac1T
$$
for some $\beta > 0$.
\end{proposition}

\begin{proof}Let $\sigma < -1/2$. We estimate
$$
\|\langle x \rangle^{-2} \an{\tau}^{-2} L(\tau) D^\sigma (\psi(t) u_0 - \psi_k(t) u_0) \|_{L^2_{\tau,x}} \; .
$$
We fix $\tau$. 

Let $R = M_k^\beta $. Above $R$, we have 
$$
\|1_{|x|\geq R}\langle x \rangle^{-2} L(\tau) D^\sigma (\psi(t) u_0 - \psi_k(t) u_0) \|_{L^2_x} \leq \langle R \rangle^{-1} \| L(\tau) D^\sigma (\psi(t) u_0 - \psi_k(t) u_0) \|_{L^\infty_x} \; .
$$
We use that $\psi(t) u_0 - \psi_k(t) u_0  = \psi'(t) u_0 - \psi_k'(t) u_0) $ and that with the choice of $u_0$, $\|\psi'_k(t) u_0\|_{H^s} \leq C \Lambda$,  $\|\psi'(t) u_0\|_{H^s} \leq C \Lambda$. We have 
$$
\| L(\tau) D^\sigma (\psi(t) u_0 - \psi_k(t) u_0) \|_{L^\infty_x} \leq \| L(\tau) D^\sigma (\psi'(t) u_0) \|_{L^\infty_x} + \| L(\tau) D^\sigma (\psi'_k(t) u_0) \|_{L^\infty_x} \; .
$$
Thanks to Sobolev embedding ($\sigma < -1/2$), we have 
$$
\| L(\tau) D^\sigma (\psi(t) u_0 - \psi_k(t) u_0) \|_{L^\infty_x} \leq C \| L(\tau)  (\psi'(t) u_0) \|_{H^s} + \| L(\tau)  (\psi'_k(t) u_0) \|_{H^s} \leq C \Lambda \; .
$$

Under $R$, we  fix  $\varepsilon < s/ \alpha$. Write $\psi(t) u_0 - \psi_k(t) u_0$ as 
$$
\psi(t) u_0 - \psi_k(t) u_0 = (1 - \Pi'_k)\psi'(t) u_0 + \Pi'_k ( \psi(t)u_0 - \psi_k(t) u_0) - (1-\Pi'_k) \psi'_k(t)u_0 
$$
where $\Pi_k'$ is the Fourier multiplier by $\eta (\frac{n}{M_k^\varepsilon})$. We have 
$$
\|D^\sigma L(\tau) (1- \Pi'_k) \psi'(t)u_0\| _{L^\infty} \leq C\|L( \tau) (1- \Pi'_k)\|_{L^2} \leq M_k^{-s\varepsilon}  \|\psi'(t)\|_{H^s}
$$
and with the norm over time 
$$
\|\an{\tau}^{-2}D^\sigma L(\tau) (1- \Pi'_k) \psi'(t)u_0\| _{L^2_t,L^\infty} \leq C \Lambda M_k^{-s\varepsilon} \; .
$$
For the same reasons,
$$
\|\an{\tau}^{-2} D^\sigma L(\tau) (1- \Pi'_k) \psi'_k(t)u_0\| _{L^2_t,L^\infty} \leq C M_k^{-s \varepsilon} \Lambda \; .
$$
For the last term, we use the finite speed propagation
$$
\|D^\sigma L(\tau) \Pi'_k ( \psi(t)u_0 - \psi_k(t) u_0)\| _{L^\infty (|x|\leq R)} \leq A+B
$$
with
$$
A= \|D^\sigma L(\tau) \Pi'_k 1_{|x|\leq R + 3\an \tau M_k^\varepsilon} ( \psi(t)u_0 - \psi_k(t) u_0)\|_{L^\infty}
$$
and 
\begin{multline*}
B = \frac1{M_k^{2\varepsilon}} \|\an{\tau}^{-3}\sup_y \|D^\sigma \Pi_k' (\psi(t) u_0 - \psi_k(t) u_0)\|_{L^\infty([y-R,y+R])}\|_{L^2_t} \leq \\
\frac1{M_k^{2\varepsilon}T} (\|\an{\tau}^{-3} \psi'(t)\|_{L^2_t,L^2(\R)} + \|\an{\tau}^{-3}\psi'_k(t)\|_{L^2_t,L^2(\pi N_k)}
\end{multline*}
(we can assume $R \leq \pi N_k$). Hence $B \leq \frac{1}{ M_k^{2\varepsilon}} \Lambda^3 T^{\gamma'}$. 

By Sobolev embedding
$$
A \leq \|  1_{|x|\leq R + 3|\tau| M_k^\varepsilon} ( \psi(t)u_0 - \psi_k(t) u_0)\|_{L^2}
$$
We use the convergence of $\psi_k(t) u_0$ towards $\psi(t)$ in $L^2$ with the weight $\langle x \rangle^{-\alpha}$, we have 
$$
\|\an{\tau}^{-2} D^\sigma L(\tau) \Pi'_k ( \psi(t)u_0 - \psi_k(t) u_0)\|_{L^2_t,L^\infty (|x|\leq R)} \leq C   \langle  R + 3|\tau| M_k^\varepsilon\rangle^{\alpha}M_k^{-s}T^{-1} \; .
$$
We use $\beta = \varepsilon$ and $\beta' = \min ( \beta, s\varepsilon , s- \alpha \varepsilon)$, we have 
$$
\|\langle x \rangle^{-2} \an{\tau}^{-2} L(\tau) D^\sigma (\psi(t) u_0 - \psi_k(t) u_0) \|_{L^2_{\tau,x}} \leq C \| \langle \tau \rangle^{\alpha-2} \|_{L^2}M_k^{-\beta'}
$$
We had chosen $\alpha $ such that $2 - \alpha > 1/2$. Hence,
$$
d(\psi_k(t) u_0 , \psi(t) u_0) \leq C M_k^{-\beta'} T^{-1}\; .
$$
\end{proof}

\begin{proposition}\label{prop-zcv}Assume that $u_0$ is such that 
$$
\| L(t)u_0\|_{Z(s)} \leq \Lambda \; .
$$
Then, for all $t \in [-T,T]$ with $T = \frac1{C \Lambda^{2\gamma}}$ and all $s'< s$, the sequence $L(\tau)\psi_k(t)u_0$ converges towards $L(\tau)\psi(t) u_0$ in $Z'(s')$ with a rate of convergence independent from $u_0$, namely
$$
\| L(\tau) (\psi(t)u_0 - \psi_k(t)) u_0\|_{Z'(s')} \leq C M_k^{-\beta} \Lambda \frac1T
$$
for some $\beta > 0$.
\end{proposition}

\begin{proof} We proceed in the same way as for Proposition \ref{prop-xmoinscv}. We can apply the local uniform proposition (\ref{prop-luc}) with $s'$. Let $R = M_k^\varepsilon$. We divide the norm between what is included above and under $R$. Above $R$, we have, 
$$
\|1_{|x|\geq R} L(\tau) \psi'(t) u_0\|_{Z'(s')} \leq \an{R}^{-\alpha} \|L(\tau) \psi'(t) u_0\|_{Z'(s')} 
$$
and then we use Strichartz estimates 
$$
\|1_{|x|\geq R} L(\tau) \psi'(t) u_0\|_{Z'(s')}\leq C\an{R}^{-\alpha} \|\psi'(t)u_0\|_{H^{s'}}  \leq C\an{R}^{-\alpha} \Lambda \; .
$$
and for the norm of $\psi'_k(t)$ we divide $\{|x|\geq R\}$ in $\cup_m [2\pi m N_k +R, 2\pi N_k (m+1) + R]\cup [-R - (m+1) 2\pi N_k, -R - m 2\pi N_k]$ to get as in Lemma \ref{lem-jenaigravemarre}
$$
\|1_{|x|\geq R} L(\tau) \psi'(t) u_0\|_{Z'(s')}\leq C (\an{R}^{-\alpha} + \an{R}^{1-\alpha}(\pi N_k)^{-1}) \; .
$$

Under $R$, we write $\psi(t) u_0 - \psi_k(t) u_0$ as 
$$
\psi(t) u_0 - \psi_k(t) u_0 = (1 - \Pi'_k)\psi'(t) u_0 + \Pi'_k ( \psi(t)u_0 - \psi_k(t) u_0) - (1-\Pi'_k)\psi_k'(t) u_0 
$$
where $\Pi_k'$ is the Fourier multiplier by $\eta (\frac{n}{M_k^\varepsilon})$. We have, by Strichartz
$$
\| L(\tau) (1- \Pi'_k) \psi'(t)u_0\|_{Z'(s')} \leq C\| (1- \Pi'_k)\psi'(t)u_0\|_{H^{s'}} \leq M_k^{(s'-s)\varepsilon}  \|\psi'(t)u_0\|_{H^s} \leq C \Lambda M_k^{(s'-s)\varepsilon} \; .
$$
By dividing $\R$ into intervals of size $2\pi N_k$, we get
$$
\|  L(\tau) (1- \Pi'_k) \psi'_k(t)u_0\|_{Z'(s')}\leq \sum_{m} \an{2\pi N_k m }^{-\alpha} \|1_{|x|\leq \pi N_k}L(\tau) (1-\Pi'_k) \psi_k'(t)u_0\|_{Z'(s')}
$$
and by using Strichartz estimates 
$$
\|  L(\tau) (1- \Pi'_k) \psi'_k(t)u_0\|_{Z(s')} \leq C \an{\pi N_k}^{-\alpha} M_k^{(s'-s)\varepsilon} \Lambda \; .
$$
For the last term, we use the finite speed propagation
$$
\|1_{|x|\leq R}L(\tau) \Pi'_k (\psi(t)u_0 - \psi'_k(t) u_0) \|_{Z'(s')} \leq A+B
$$
with
$$
A= \| L(\tau) \Pi'_k 1_{|x|\leq R+ 3 \an \tau M_k^{\varepsilon}} (\psi'(t)u_0 - \psi'_k(t) u_0) \|_{Z'(s')}
$$
and 
$$
B = \frac1{M_k^{2\varepsilon} }\|  \an{\tau}^{-3} \sup_y \Big(\| \psi'(t) u_0 - \psi'_k (t) u_0\|_{L^\infty([y-R,y+R])} + \| \psi'(t) u_0 - \psi'_k (t) u_0\|_{H^{s' }([y-R,y+R])}\Big) \|_{L^2_\tau} \; .
$$
For $A$, we use again Strichartz estimates 
$$
A \leq \| \an{\tau}^{-2} 1_{|x|\leq R+ 3|\tau|M_k^{\varepsilon}} (\psi'(t)u_0 - \psi'_k(t) u_0) \|_{L^p_\tau,H^{s'}} \; .
$$
For $B$, we use the local bounds to get
$$
B \leq \frac{1}{M_k^{2\varepsilon}} \Lambda \; .
$$

Thanks to Proposition \ref{prop-luc}, we have
$$
A \leq \|\an{\tau}^{-2} (R+ 3\an \tau M_k^{\varepsilon})^{\alpha}\|_{L^p_t} \|\an{x}^{-\alpha} D^{s'}(\psi'(t)u_0 - \psi'_k(t) u_0) \|_{L^2_x} \leq CM_k^{\varepsilon \alpha} M_k^{s'-s} \Lambda T^{-1}\; .
$$

We choose $\varepsilon < (s-s') / \alpha$ and $\beta = \min (\varepsilon, s-s' - \alpha \varepsilon, (s-s') \varepsilon)$ and also $\alpha $ such that $2 - \alpha > 1/p$ to be able to integrate in time and conclude.  \end{proof}

\subsection{local continuity}

In this subsection, we prove that the flows $\psi_k$ are locally continuous in the initial datum.

\begin{proposition}\label{prop-loccont} Let $u_{0,1}$ and $u_{0,2}$ such that
$$
\| L(t) u_{0,i}\|_{Z(s)} \leq \Lambda \; .
$$
Then,
$$
\|\psi'_k(t)u_{0,1}- \psi'_k(t)u_{0,2}\|_{Y_T(s,\pi N_k)} \leq C\| L(t)(u_{0,1}-u_{0,2})\|_{Z(s)}
$$
with $C$ independent from $s, \Lambda ,k$.
\end{proposition}

\begin{proof}
Let $v_i = \psi'(t) u_{0,i}$. We have 
$$
v_1- v_2 = i\int_{0}^t \Pi_k P_k \chi^{1/3}\Big( |L(\tau) u_{0,1}+ v_1|^2(L(\tau) u_{0,1}+ v_1) - |L(\tau) u_{0,2}+ v_2|^2(L(\tau) u_{0,2}+ v_2) \Big)d\tau
$$
and using Strichartz estimates
$$
\|v_1 - v_2\|_{Y_T(s, \pi N_k)} \leq C T^{\gamma'} \Lambda^2\Big(\|\chi^{1/3}L(t)(u_{0,1}-u_{0,2})\|_{Z'(s)} + \|v_1 - v_2\|_{Y_T(s, \pi N_k)}\Big)\; .
$$
As $T = \frac1{C\Lambda^{2\gamma}}$ with $C$ big enough, we get the result.
\end{proof}

\begin{proposition} For all $u,v$ such that
$$
\| L(t) u\|_{Z(s)} \leq \Lambda \; .
$$
and
$$
\| L(t) v\|_{Z(s)} \leq \Lambda \; .
$$
and all $t\in [-T,T]$ with $ T =\frac1{C\Lambda^{2\gamma}}$, we have 
$$
d(\psi_k(t)u, \psi_k(t) v )\leq C \|  L(\tau) (u-v)\|_{Z'(s)} + d(u,v)
$$
and 
$$
\|  L(\tau) (\psi_k(t)u - \psi_k(t)v)\|_{Z'(s)} \leq C \|  L(\tau) (u-v)\|_{Z'(s)} \; .
$$
\end{proposition}

\begin{proof}We start with $Z'(s)$. We write $\psi_k(t) u - \psi_k(t) v  = \psi'_k(t)u - \psi'_k(t) v + L(t) (u-v)$. We use that $\psi'(t)u$ and $\psi'(t)v$ are in $E_k$ and we divide the space line into intervals of size $2\pi N_k$ and the time line into intervals of size $1$ and Strichartz estimates to get
$$
\|  L(\tau) (\psi'_k(t)u - \psi'_k(t)v)\|_{Z'(s)} \leq C (\pi N_k)^{-1} \|\psi'_k(t)u - \psi'_k(t)v\|_{H^s} \leq C \| L(\tau) (u-v)\|_{Z'(s)} \; .
$$
We use a change of variable on time and the fact that $\an{\tau}^2 \an{t+\tau}^{-2}$ is bounded to get
$$
\| L(\tau) (L(t)u - L(t)v)\|_{Z'(s)} \leq   C(t) \| L(\tau) (u-v)\|_{Z'(s)} \; .
$$

For $d$, we use the same arguments, replacing Strichartz estimates by Sobolev embeddings
\begin{multline*}
\|\an{x}^{-2} \an{\tau}^{-2} D^\sigma L(\tau) (\psi'_k(t)u - \psi'_k(t)v)\|_{L^2} \leq C (\pi N_k)^{-1} \|D^ \sigma L(t)(\psi'_k(t)u - \psi'_k(t)v)\|_{L^\infty} \\
\leq C \| L(\tau) (u-v)\|_{Z'(s)} \; .
\end{multline*}
and a change of variable to get
$$
\|\an{x}^{-2} \an{\tau}^{-2} D^\sigma L(\tau) (L(t)u - L(t)v)\|_{L^2} \leq C\an{t}^2\|\an{x}^{-2}\an{\tau}^{-2} D^\sigma L(\tau) (u-v)\|_{L^2}
$$
as in the proof of the continuity of $L(t)$ to get
$$
d(\psi_k(t)u, \psi_k(t) v )\leq C(t) \|L(\tau) (u-v)\|_{Z'(s)} + d(u,v) \; .
$$

\end{proof}

\section{Global analysis}

In this section, we extend results of the previous section to global times, assuming that we take the initial data in a given set $A$. In this way, we prove the global well-posedness of \eqref{oureq}, but also the global uniform convergence of the flow $\psi_k(t)$ of our approaching equations \eqref{approach} towards the flow $\psi(t)$ of \eqref{oureq}. We do this by propagating these properties from time to time. The global well-posedness of the approaching equations are due to energy estimates and the fact that the non linear part is in $E_k$, a finite dimensional space.

\subsection{Global well-posedness of approaching equations}

\begin{proposition}\label{prop-gwppsik}
The Cauchy problem
\begin{equation}\label{appeq}
\begin{cases}
i\partial_tv+\Delta v-\Pi_kP_k\left(\chi|\Pi_k(L(t)u_0+v)|^2(\Pi_k(L(t)u_0+v)\right)=0\\
v(x,0)=0.
\end{cases}
\end{equation}

is globally well posed in $L(t)u_0+E_k$.

\end{proposition}

\begin{proof}
We define the energy of the system to be
\begin{equation}
\mathcal{E}_k(v)=\frac12\int_{-\pi N_k}^{\pi N_k}\left(|\nabla v|^2+\frac12P_k\left(\chi |R_k(u_0,v)|^4\right)\right)
\end{equation}
where we are denoting with $R_k(u_0,v)=\Pi_k(L(t)u_0 + v)$.
Differentiating in time gives
$$
\partial_t \mathcal{E}_k(v)=-{\rm Re}\int_{-\pi N_k}^{\pi N_k}\partial_t\overline{v}\Delta v+{\rm Re}\int_{-\pi N_k}^{\pi N_k} \chi|R_k(u_0,v)|^2R_k(u_0,v)\partial_t\overline{R_k(u_0,v)}.
$$
Since $\partial_tR_k(u_0,v)=\Pi_k(\partial_t L(t)u_0+\partial_tv)$  and $\Pi_k$ is a self-adjoint operator we can rewrite it as

\begin{eqnarray*}
\partial_t \mathcal{E}_k(v) & = & {\rm Re}\left(\int_{-\pi N_k}^{\pi N_k}\partial_t\overline{v}\:(-\Delta v+P_k\left(\chi|R_k(u_0,v)|^2(R_k(u_0,v))\right)\right) \\
& + & {\rm Re}\left(\int_{-\pi N_k}^{\pi N_k}\partial_t\overline{L(t)u_0} \left(P_k\chi|R_k(u_0,v)|^2R_k(u_0,v)\right)\right) \\
& = & I+II.
\end{eqnarray*}

It is immediate to verify, since $v$ solves equation \eqref{appeq}, that
\begin{equation*}
I={\rm Re}\left((-i)\int_{-\pi N_k}^{\pi N_k}\overline{\partial_tv}\partial_t v\right)=0,
\end{equation*}
so that we are left with $II$. We estimate it as follows
\begin{eqnarray*}
II&\leq&\|\Pi_k\partial_tL(t)u_0\|_{L^4(\pi N_k)}\|P_k\chi|R_k(u_0,v)|^2R_k(u_0,v)\|_{L^{4/3}(\pi N_k)}
\\
&\leq&\|\Pi_k\Delta L(t)u_0\|_{L^4(\pi N_k)}\|P_k\chi\|_{L^\infty}^{1/3}\|P_k\chi^{1/4}R_k(u_0,v)\|^3_{L^4}
\\
&\lesssim& \|\Pi_k\Delta L(t)u_0\|_{L^4(\pi N_k)}\mathcal{E}_k(v)^{3/4}
\end{eqnarray*}
since $\chi\in L^\infty$. 
Let us now consider the quantity $\mathcal{H}_k(v)=\mathcal{E}_k(v)^{1/4}$: we have
\begin{equation*}
\partial_t\mathcal{H}_k(v)=\frac14\left(\partial_t\mathcal{E}_k(v)\right)\mathcal{E}_k(v)^{-3/4}
\lesssim \|\Pi_k\Delta L(t)u_0\|_{L^4(\pi N_k)}
\end{equation*}
which then gives by integration in time (the initial datum is $0$)
\begin{equation*}
\mathcal{H}_k(v)\lesssim\int_0^t\|\Pi_k\Delta L(\tau)u_0\|_{L^4(\pi N_k)}d\tau.
\end{equation*}
The boundedness of $\int_0^t\|\Pi_k\Delta L(\tau)u_0\|_{L^4(\pi N_k)}d\tau$ ($\Pi_k$ is a smooth cutoff thus the presence of the Laplacian is not a problem) then guarantees that $\mathcal{H}_k(v)=\mathcal{E}_k(v)^{1/4}$ is bounded for every time $t$. As a consequence, also $\mathcal{E}_k(v)^{1/2}$ is bounded for every time $t$. We now exploit the fact that $v$ belongs to $E_k$ which is of finite dimension, and which can be equipped with $H^s$ norm for every $s>0$; we can therefore estimate, for every $s\in[s_\infty,s_0]$,
\begin{equation*}
\|v\|_{H^s(\pi N_k)}\leq C_k\|v\|_{H^1(\pi N_k)}.
\end{equation*}
Since $\mathcal{E}_k(v)^{1/2}$ controls $\|v\|_{H^1(\pi N_k)}$, and the $Z_{t_0}(s)$ norm of $L(t)u_0$ is uniform bounded, the proof is concluded.

\end{proof}

\subsection{Global continuity of approaching flows}
 
In this subsection, we prove that for a certain set of initial data, the flow $\psi_k(t)$ is globally continuous with respect to the initial datum.

\begin{proposition}\label{prop-globcont} Let $t \in \R_+$ and $R \geq 0$. Let $B_t(R)$ be the set of $u$ such that
$$
\|L(\tau) u \|_{L^4([0,t], H^s(\pi N_k))} \leq R \; , \;  \|\an{x}^{-\alpha} L(\tau) u \|_{L^4([0,t],  L^\infty)} \leq R \; , \;  \|\an{x}^{-\alpha}D^s L(\tau) u \|_{L^4([0,t],  L^2)} \leq R 
$$
with $s \in ]1/3,1/2[$
For all $t' \in [0,t]$ the map $\psi'(t')$ is continuous in $B_t(R)$ for the $X^{-1/2-}$ topology.
\end{proposition}

\begin{proof} Let $T$ be such that $T = \frac1{C_k f(t,R)}$ with $C_k$ a big enough constant and $f(t,R)$ defined as follows. Thanks to the proof of Proposition \ref{prop-gwppsik}, we know that the $L^2$ norm of $\psi'(t') u$ is bounded by some function of $R$ and $t$, $g(t,R)$ for all $u \in B_t(R)$ and $t' \in [0,t]$, we set $f(t,R) = \max (R, g(t,R))^{4}$. Let $t_n = nT$, we prove by induction on $n$ that $\psi'(t')$ is continuous for $t' \in [t_{n-1},t_n]$ in $L^\infty([t_{n-1}, t_n], L^2(\pi N_k))$. 

For $n=0$, we have $t_0 = 0$ and $\psi'(0) = Id$. 

For the induction $n\rightarrow n+1$, we consider that $u = \psi'(t') u_0$ and $v = \psi'(t') v_0$ with $u_0 , v_0 \in B_t(R)$ are the fixed points
$$
u(t') = L(t'-t_n) u(t_n) + i \int_{t_n}^{t'} \Pi_k P_k \chi \Big(|u(\tau) +L(\tau) u_0|^2(u(\tau) + L(\tau) u_0) \Big) d\tau 
$$
$$
v(t') = L(t'-t_n) v(t_n) + i \int_{t_n}^{t'} \Pi_k P_k \chi \Big(|v(\tau) +L(\tau) v_0|^2(v(\tau) + L(\tau) v_0) \Big) d\tau \; .
$$
Hence we get that, since $H^s$ is embedded in $L^6$, and since $u,v$ belong to $E_k$ which is of finite dimension and hence every norm is equivalent on $E_k$:
\begin{multline*}
\|u(t') - v(t') \|_{L^2}  \leq \|u(t_n) - v(t_n)\|_{L^2} + C_k \|L(\tau) (u_0 - v_0)\|_{L^2([t_n, t_{n+1}], L^6)} g(t,R)^2 + \\
C_k T^{1/2} g(t,R)^2 \|u -v\|_{L^\infty([t_n, t_{n+1}], L^2)} \; .
\end{multline*}
With our definition of $T$, $ C_k T^{1/2} g(t,R)^2  < 1$. Therefore,
$$
\|u(t') - v(t') \|_{L^2}  \leq C_k \|u(t_n) - v(t_n)\|_{L^2} + C_k \|L(\tau) (u_0 - v_0)\|_{L^2([t_n, t_{n+1}], L^6(\pi N_k))} g(t,R)^2 \; .
$$
We focus on
$$
\|L(\tau) (u_0 - v_0)\|_{ L^6(\pi N_k))} = \|P_k L(\tau) (u_0 - v_0)\|_{ L^6(\pi N_k))} \; .
$$
By Sobolev embedding, we have 
$$
\|L(\tau) (u_0 - v_0)\|_{ L^6(\pi N_k))} \leq \|P_k L(\tau) (u_0 - v_0)\|_{H^{1/3}(\pi N_k))} \; .
$$
We introduce the smooth frequency cut-off $\Pi_r$, we have 
$$
\|L(\tau) (u_0 - v_0)\|_{ L^6(\pi N_k))} \leq \|(1-\Pi_r) P_k L(\tau) (u_0 - v_0)\|_{H^{1/3}(\pi N_k))} + \|\Pi_r P_k L(\tau) (u_0 - v_0)\|_{H^{1/3}(\pi N_k))} \; .
$$
We have
$$
\|(1-\Pi_r) P_k L(\tau) (u_0 - v_0)\|_{H^{1/3}(\pi N_k))} \leq C r^{1/3 -s} R
$$
and since $E_k$ is of finite dimension
$$
\|\Pi_k P_k L(\tau) (u_0 - v_0)\|_{L^2([t_{n}, t_{n+1}],H^{1/3}(\pi N_k))} \leq C_k (t) \|\an{\tau}^{-2} D^\sigma \Pi_k P_k L(\tau) (u_0 - v_0) \|_{L^2_t,x} \leq C_k(t,r) d(u_0,v_0) \; .
$$
We use that $\psi(t_n)$ is continuous by induction hypothesis and that $r^{s-1/3}$ goes to $0$ when $r \rightarrow \infty$ to conclude.\end{proof}

\subsection{Global uniform convergence}

\begin{definition} For all $n\in \mathbb{N}$, let
$$
\Lambda_n = (1+n)^{\gamma'/4} \Lambda, T_n = \frac1{C\Lambda_n^{2\gamma}} = \frac1{C\Lambda^{2\gamma}\sqrt{1+ n}}, t_n = \sum_{k=1}^n T_k \sim \sqrt n, s_n = s_\infty + \frac1n (s_0 - s_\infty)\; .
$$
Let
$$
A_{k,n} (\Lambda) = \{ u_0 |\; \| L(t) \psi_k(t_n)u_0\|_{Z(s_n)} \leq \Lambda_{n+1}\}\; ,
$$
$$
A_k(\Lambda) = \bigcap_n A_{k,n}(\Lambda) \mbox{ and } A(\Lambda) = \limsup_k A_k(\Lambda) \; .
$$
\end{definition}

\begin{proposition}\label{prop-guc}
For all $n\in \mathbb{N}$, we have \begin{enumerate}
\item for all $t \in [0,t_n]$, all $u_0 \in A(\Lambda)$, $\psi(t)u_0$ is well-defined,
\item for all $t \in [0,t_n]$, $\varepsilon > 0$, there exists $k_0 \in \mathbb{N}$ such that for all $k\geq k_0$ and all $u_0 \in A(\Lambda)$, 
$$
d(\psi(t) u_0 ,\psi_k(t) u_0) \leq \varepsilon \mbox{ and } \| L(\tau)(\psi(t) u_0 - \psi_k(t) u_0)\|_{Z'(s_n)}\leq \varepsilon \; ,
$$
\item $\|L(t) \psi(t_n) u_0\|_{Z(s_n)} \leq \Lambda_{n+1}$,
\item there exists $k_1 \in \mathbb{N}$ such that for all $k\geq k_1$ and all $u_0 \in A(\Lambda)$, 
$$\|L(t) \psi_k(t_n)u_0\|_{Z(s_n)} \leq 2\Lambda_{n+1}.$$
\end{enumerate}
\end{proposition}

\begin{proof}
We assume $t\geq0$ and proceed by induction to prove the property $P(n)$. The inductive base, that is $(P(n=0)$), is easily done: properties $1$, $2$ and $3$ are trivial indeed, as is property $4$ (which is just $\|L(t)u_0\|_{Z(s_0)}\leq 2\Lambda_1$) once noticed that $u_0\in A_{k,0}(\Lambda)$ for some $k$ big enough. 

Let us now prove the inductive step, that is $P(n)\Rightarrow P(n+1)$; we take $\tau=t_n+t\in[t_n,t_{n+1}]$.
We start with property $1$. We use $3$ at the $n$-th step to write
\begin{equation}\label{uuu}
\|L(t)\psi(t_n)u_0\|_{Z(s_n)}\leq \Lambda_{n+1}
\end{equation}

We now rely on our local well-posedness theory: assumption \eqref{uuu} implies indeed that $\psi(t_n+\tau)u_0$ is well defined for $|\tau|\leq T_{n+1}$ and so that $\psi(t)u_0$ is well defined in $[0,t_n+T_{n+1}]=[0,t_{n+1}]$ (see Proposition \ref{prop-lwp}).

To prove property $2$ we make use of both our local continuity and local uniform convergence results. We start writing
$$
\psi(t)u_0-\psi_k(t)u_0=[\psi(\tau)\psi(t_n)u_0-\psi_k(\tau)\psi(t_n)u_0]+[\psi_k(\tau)\psi(t_n)u_0-\psi_k(\tau)\psi_k(t_n)u_0]
$$
$$
=I_1+I_2
$$
to deal separately with $I_1$ and $I_2$.
In order to bound $I_1$, we use the local uniform convergence, namely Propositions \ref{prop-xmoinscv} and \ref{prop-zcv} to have that, for every $\tau\in[t_n-T_{n+1},t_n+T_{n+1}]$ and some $\beta>0$,
\begin{equation*}
d(\psi(\tau)\psi(t_n)u_0,\psi_k(\tau)\psi(t_n)u_0)\leq M_k^{-\beta}\Lambda_{n+1}^{1-2\gamma}
\end{equation*}
and for all $s'_{n+1}<s_{n+1}$
\begin{equation*}
\|L(t)(\psi(\tau)\psi(t_n)u_0-\psi_k(\tau)\psi(t_n)u_0)\|_{Z'(s_n')}
\leq M_k^{-\beta}\Lambda_{n+1}^{1-2\gamma}.
\end{equation*}
Therefore, provided $k\geq k_0$ is big enough, we can make both the right hand sides above smaller than any $\varepsilon>0$.

Turning to $I_2$, we rely on local continuity of $\psi$: properties $3$ and $4$ at the $n$-th step guarantee again the sufficient bounds on the initial datum, which here are $\psi(t_n)$ and $\psi_k(t_n)$, and so the application of Proposition \ref{prop-loccont} yields, for all $\tau\in[t_n-T_{n+1},t_n+T_{n+1}]$,
\begin{equation*}
d(\psi_k(\tau)\psi(t_n)u_0,\psi_k(\tau)\psi_k(t_n)u_0)
\end{equation*}
\begin{equation*}
\leq C\|L(t)(\psi(t_n)u_0-\psi_k(t_n)u_0)\|_{Z'(s_n)}+d(\psi(t_n)u_0,\psi_k(t_n)u_0),
\end{equation*}
and
\begin{equation*}
\|L(t)(\psi_k(\tau)\psi(t_n)u_0-\psi_k(\tau)\psi_k(t_n)u_0)\|_{Z'(s_n)}\leq
C \|L(t)(\psi(t_n)u_0-\psi_k(t_n)u_0)\|_{Z'(s_n)}.
\end{equation*}
Since $P(n)$ is supposed to be true, in particular $2$ holds for $t=t_n$; we can therefore estimate both the right hand sides above (and so $I_2$) with $\epsilon$.

Property $3$ immediately follows from the fact that for an increasing sequence $k_m\rightarrow+\infty$ we have from the assumption $u_0\in A_{k_m,n+1}(\Lambda)$
\begin{equation*}
\|L(t)\psi_{k_m}(t_{n+1})u_0\|_{Z(s_{n+1})}\leq \Lambda_{n+2};
\end{equation*}
weak convergence of $\psi_k(t)u_0$ towards $\psi(t)u_0$ then yields $3$.

Finally, to prove property $4$ it is enough to notice that convergence in $Z(s)$ norm is implied by convergence in $Z'(s)$ norm, which is guaranteed by Proposition 3.5; the estimate on the limit provided by property $3$ then guarantees the desired bound on $\psi_k(t_n)$ for $k\geq k_1$ for some $k_1$ big enough.
\end{proof}

\section{Conclusion : Invariance of \texorpdfstring{$\rho$}{rho} under \texorpdfstring{$\psi(t)$}{psi(t)}}

In this section, we prove the final part of Theorem \ref{theo-result}. We begin by proving that the set of initial data is of full $\rho$-measure ($\rho(A) = 1$), and then, we prove the invariance of $\rho$ under $\psi(t)$.

\subsection{Measure of \texorpdfstring{$A$}{A}}

\begin{lemma} For all $k$ and all $\Lambda \geq 1$, we have 
$$
\rho_k(A_k(\Lambda)^c) \leq \frac{C}{\Lambda^{8\gamma}} \; .
$$
\end{lemma}

\begin{proof}We have 
$$
A_k(\Lambda)^c = \bigcup_n(A_{k,n}(\Lambda)^c)
$$
thus
$$
\rho_k (A_k(\Lambda)^c) \leq \sum_n \rho_k(A_{k,n}(\Lambda)^c) \; .
$$
The set $A_{k,n}((\Lambda)^c) $ is given by
$$
\psi_k(t_n)^{-1}(\{ u \; |\; \| L(\tau) u \|_{Z(s_n)} > \Lambda_{n+1}\})
$$
and as $\rho_k$ is invariant under $\psi_k(t_n)$,
$$
\rho_k ( A_{k,n}^c) \leq C e^{-c \Lambda_{n+1}^2} =  Ce^{-c (n+1)^{\gamma'/4}\Lambda^2}
$$
and summing gives the desired bound.
\end{proof}

\begin{lemma}We have for all $\Lambda$
$$
\rho(A(\Lambda)^c) \leq \frac{C}{\Lambda^{8\gamma}} \; .
$$
\end{lemma}

\begin{proof}
First, we use that $A(\Lambda)^c = \liminf_k A_{k} (\Lambda)^c$ to get, thanks to Fatou's lemma,
$$
\rho(A(\Lambda)^c) \leq \liminf_k \rho(A_k(\Lambda)^c) \; .
$$
We have that 
$$
A_{k,n(\Lambda)^c)} \subseteq \{ u \; |\; \| L(\tau) \psi'_k(t_n) u \|_{Z(s_n)} > \Lambda_{n+1}\}\cup \{ u \; |\; \| L(\tau + t_n) u \|_{Z(s_n)} > \Lambda_{n+1}\}
$$
thus we get
$$
A_{k,n}(\Lambda)^c) \subseteq  \bigcup_n \{ u \; |\; \| L(\tau) \psi'_k(t_n) u \|_{Z(s_n)} > \Lambda_{n+1}/2\}\cup \bigcup_n\{ u \; |\; \| L(\tau + t_n) u \|_{Z(s_n)} > \Lambda_{n+1}/2\}\; .
$$
Thanks to the definition of $\mu$ and $\rho$ we have that 
$$
\rho\Big( \bigcup_n\{ u \; |\; \| L(\tau + t_n) u \|_{Z(s_n)} > \Lambda_{n+1}/2\}\Big) \leq \frac{C}{\Lambda^{8 \gamma}} \; .
$$
We reuse the sets $B_t(R)$ defined in Proposition \ref{prop-globcont}. As the support of $\rho$ is included in $\cup_R B_t(R)$ we have 
$$
\rho (\bigcup_n \{ u \; |\; \| L(\tau) \psi'_k(t_n) u \|_{Z(s_n)} > \Lambda_{n+1}/2\}) = \sup_R \rho_{t_n,R} (\bigcup_n \{ u \; |\; \| L(\tau) \psi'_k(t_n) u \|_{Z(s_n)} > \Lambda_{n+1}/2\})
$$
where $\rho_{t_n, R}(A)  = \rho(A \cap B_{t_n}(R))$.

Since for all $F$ bounded and Lipschitz continuous of $X^{-1/2-}$, 
$$
\int |F\circ \varphi - F\circ \phi_k| 1_{B_{t_n}(R)}d\mathbb P \leq \int |F\circ \varphi - F\circ \phi_k| d\mathbb P \rightarrow_{k\rightarrow \infty} 0
$$
we get that the sequence $\rho_{t_n, R, k} = \rho_k( \cdots \cap B_{t_n}(R))$ converges weakly towards $\rho_{t_n, R}$. 

As $\psi'(t_n)$ is continuous on $B_{t_n}(R)$ from $X^{-1/2-}$ to $E_k$ with any norm, we get that 
$$
\{ u \; |\; \| L(\tau) \psi'(t_n) u \|_{Z(s_n)} > \Lambda_{n+1}/2\}
$$
is open and hence 
$$
\bigcup_n\{ u \; |\; \| L(\tau) \psi'(t_n) u \|_{Z(s_n)} > \Lambda_{n+1}/2 \}
$$
is open too for the trace topology of $X^{-1/2-}$ over $B_{t_n}(R)$. Therefore
\begin{multline*}
\rho_{t_n, R}(\bigcup_n\{ u \; |\; \| L(\tau) \psi'_k(t_n) u \|_{Z(s_n)} > \Lambda_{n+1}/2\}) \leq \\
\liminf_j \rho_{t_n, R,j}\bigcup_n\{ u \; |\; \| L(\tau) \psi'_k(t_n) u \|_{Z(s_n)} > \Lambda_{n+1}/2\}\; .
\end{multline*}
As $\liminf_{j,k} r_{j,k} \leq \liminf_k r_{k,k}$ and $\rho_{t_n, R,j} \leq \rho_j$ we get
$$
\rho_{t_n,R}(A(\Lambda)^c)  \leq C \Lambda^{-8\gamma}+ \liminf_k  \rho_{k}\bigcup_n\{ u \; |\; \|\xi L(\tau) \psi'_k(t_n) u \|_{Z(s_n)} > \Lambda_{n+1}/2\}\; .
$$
Finally, we use the previous lemma, the fact that 
$$
\bigcup_n\{ u \; |\; \| L(\tau) \psi'_k(t_n) u \|_{Z(s_n)} > \Lambda_{n+1}/2\} \subset A_k(\Lambda/4)^c \cup \bigcup_n\{ u \; |\; \|\xi L(\tau + t_n) u \|_{Z(s_n)} > \Lambda_{n+1}/4\}
$$
and we take the supremum over $R$ to get
\begin{eqnarray*}
\rho(A(\Lambda)^c ) &\leq & \liminf_k \Big( \rho_k(A_k(\Lambda/4)^c ) + \rho_k \Big( \bigcup_n \Big\{ u \Big| \|\xi L(\tau + t_n )u\|_{Z(s_n)}\geq \Lambda_{n+1}/4\Big\} \Big) \Big)\\
  &\leq & C \Lambda^{-8 \gamma'}\; .
\end{eqnarray*}
\end{proof}

\begin{proposition} The set $A$ is of full $\rho$-measure. \end{proposition}

\subsection{Invariance of the measure}

\begin{proposition}
For every $A\subset X^{-1/2-}$ measurable and all time $t$ we have
\begin{equation}\label{invfinal}
\rho(\psi(t)^{-1}A)=\rho(A).
\end{equation}
\end{proposition}

\begin{proof}
The proof will make use of all the results we have proved in the paper. It will be enough to show the invariance of the measure for every $K\subset X^{-1/2-}$ closed set, since closed sets generate the topological $\sigma$-algebra of $X^{-1/2-}$.

For a fixed $K\subset X^{-1/2-}$ closed set, let us define for every $\varepsilon>0$ the open set
$$
K_\varepsilon=\{u\in X^{-1/2-}|\exists v\in K:d(u,v)<\varepsilon\}.
$$
Continuity of $\psi(t)$ on $X^{-1/2-}$ (see Lemma \ref{contl} and Proposition \ref{prop-loccont}) imply that $\psi(t)^{-1}(K_\varepsilon)$ is open. We decompose $K_\varepsilon=(K_\varepsilon\cap A(\Lambda)\cup (K_\varepsilon\cap A(\Lambda)^c)$ so that $K_\varepsilon\subset(K_\varepsilon\cap A(\Lambda))\cup A(\Lambda)^c$, and use weak convergence of $\rho_k$ towards $\rho$ (Proposition \ref{prop-cvm}) to write

\begin{eqnarray}\label{minn}
\nonumber
\rho(\psi(t)^{-1}(K)) & \leq &\rho(\psi(t)^{-1}(K_\varepsilon)) \\
& \leq & \liminf_{k\rightarrow+\infty} \left(\rho_k(\psi(t)^{-1}(K_\varepsilon)\cap A(\Lambda)) +\rho_kA(\Lambda)^c\right).
\end{eqnarray}
By construction $\displaystyle A(\Lambda)^c=\liminf_{j\rightarrow+\infty} A_j(\Lambda)$ and so $\rho_k(A(\Lambda))^c\leq \displaystyle\liminf_{j\rightarrow+\infty}\rho_k(A_j(\Lambda)^c)$; this yields
\begin{eqnarray*}
\rho(\psi(t)^{-1}(K))&\leq& \liminf_{k\rightarrow+\infty}\liminf_{j\rightarrow+\infty}\left(\rho_k(\psi(t)^{-1}(K_\varepsilon)\cap A(\Lambda))+\rho_k(A_j(\Lambda)^c)\right)
\\
&\leq& 
\liminf_{k\rightarrow+\infty}\left(\rho_k(\psi(t)^{-1}(K_\varepsilon)\cap A(\Lambda))+\rho_k(A_k(\Lambda)^c)\right)
\end{eqnarray*}
where in the last step we have used that 
$\displaystyle\liminf_{k\rightarrow+\infty}\liminf_{j\rightarrow+\infty}P_{j,k}\leq\liminf_{k\rightarrow+\infty}P_{k,k}$. We deal separately with the two terms above. First of all, we use Lemma 5.1 to estimate
\begin{equation}\label{term1}
\displaystyle\rho_k(A_k(\Lambda)^c)\leq\frac{C}{\Lambda^{8\gamma}}
\end{equation}
(notice that the bound is independent on $k$). For the other term,  we rely on the uniform convergence of $\psi_k$ towards $\psi$ in $A(\Lambda)$, namely on point $2$ of Proposition \ref{prop-guc} which we recall to state that for every $\varepsilon>0$ there exists $k_0\in\mathbb{N}$ such that for all $k\geq k_0$ and all $u_0\in A(\Lambda)$, $d(\psi(t)u_0,\psi_k(t)u_0)\leq \varepsilon$. Using this fact, we have that for all $u\in\psi(t)^{-1}(K_\varepsilon\cap A(\Lambda))$ and all $k\geq k_0$ there exists $\tilde{u}\in K$ such that  $d(\psi(t)u,\tilde{u})<\varepsilon$ and
\begin{equation*}
d(\psi_k(t)u,\tilde{u})\leq d(\psi_k(t)u,\psi(t)u)+d(\psi(t)u,\tilde{u})<2\varepsilon,
\end{equation*}
which means that, for all $k\geq k_0$,
\begin{equation*}
\psi(t)^{-1}(K_\varepsilon)\cap A(\Lambda)\subset \psi_k(t)^{-1}(K_{2\varepsilon}).
\end{equation*}
This provides for the second term in \eqref{minn} the estimate
\begin{equation}\label{term2}
\rho_k(\psi(t)^{-1}(K_\varepsilon)\cap A(\Lambda))\leq \rho_k(\psi_k(t)^{-1}(K_{2\varepsilon})).
\end{equation}
Plugging \eqref{term1} and \eqref{term2} into \eqref{minn} and using invariance of $\rho_k$ under the flow $\psi_k$ (Proposition \ref{prop-fininvar}) yields
\begin{equation*}
\rho(\psi(t)^{-1}(K))\leq \liminf_{k\rightarrow+\infty}\rho_k(\psi_k(t)^{-1}(K_{2\varepsilon}))+\frac{C}{\Lambda^{8\gamma}}
\leq \liminf_{k\rightarrow+\infty}\rho_k(K_{2\varepsilon})+\frac{C}{\Lambda^8\gamma}
\end{equation*}
We now use that $\liminf\leq\limsup$ and that $K_{2\varepsilon}\subset \overline {K_{2\varepsilon}}$ to estimate further with
\begin{equation*}
\rho(\psi(t)^{-1}(K))\leq \limsup_{k\rightarrow+\infty}\rho_k(\overline{K_{2\varepsilon}})+\frac{C}{\Lambda^{8\gamma}}.
\end{equation*}
Since $\overline{K_{2\varepsilon}}=\{u\in X^{-1/2-}| d(u,K)\leq 2\varepsilon\}$ is closed we have again, by weak convergence of $\rho_k$, that $\displaystyle\limsup_{k\rightarrow+\infty}\displaystyle\rho_k(\overline{K_{2\varepsilon}})\leq\rho (\overline{K_{2\varepsilon}})$. Letting $\varepsilon\rightarrow 0$, dominated convergence theorem gives
\begin{equation*}
\rho(\psi(t)^{-1}(K))\leq \rho(K)+\frac{C}{\Lambda^{8\gamma}};
\end{equation*}
letting also $\Lambda\rightarrow+\infty$ finally gives
\begin{equation*}
\rho(\psi(t)^{-1}(K))\leq \rho(K).
\end{equation*}
On the other hand, the continuity and reversibility of the flow $\psi(t)$ on $X^{-1/2-}$ imply
\begin{equation*}
\rho(K)\leq\rho(\psi(-t)^{-1}\psi(t)^{-1}(K))\leq \rho(\psi(t)^{-1}(K))
\end{equation*}
which concludes the invariance for the measure $\rho$ on closed sets of $X^{-1/2-}$. Since closed sets generate the topological $\sigma$-algebra of $X^{-1/2-}$ the proof is concluded.
\end{proof}

\appendix

\section{Finite propagation speed}
In this appendix, we prove the property related to finite propagation speed that we used throughout the paper. It gives estimates on the influence that have different parts of the initial datum $u_0$ on the restricted $1_{|x|\leq R} L(t) u_0(x)$ at some fixed time $t$.

\begin{proposition} Let $R \geq 0$, $t\in \R^*$ and $T \geq |t|$ and $p\geq 1$, for all $f$ we have 
$$
\|1_{|x|\leq R} L(t)\Pi_k f\|_{L^p} \lesssim \|1_{|x|\leq R}L(t) \Pi_k 1_{|y|\leq R+3M_kT} f(y)\|_{L^p} + \frac{1}{M_k^2 T}  \sup_y \|f\|_{L^p([y-R,y+R])}\; .
$$
\end{proposition}

\begin{proof} First, we prove that for $|z|\geq 3 M_k T$ with $T\geq |t|$, we have
$$
|K_t (z)|\leq C_\eta \frac1{M_k} \frac1{z^2}\; .
$$
By definition, we have 
$$
K_t(z) = \int_{\R} e^{i n^2 t +in z} \eta_k(n) dn = \int_{\R} (2in t + iz ) e^{i n^2 t +in z} \frac{\eta_k(n)}{2in t + iz}dn
$$
where $\eta_k(n) = \eta \Big( \frac{n}{M_k}\Big)$. Since $|z| \geq 3 M_k T$, we have $|2in t +iz| \geq \frac13 |z| \geq 3M_k T > 0$. Let $f(n) =  \frac{\eta_k(n)}{2in t + iz}$. the function $f$ is $\mathcal C^\infty$ with compact support hence by integration by part
$$
K_t(z) = -\int_{\R}  e^{i n^2 t +in z} f'(n) dn \; .
$$
With a second integration by parts, we have
$$
K_t(z) = \int_{\R}  e^{i n^2 t +in z} g'(n) dn 
$$
with $g(n) =  \frac{f'(n)}{-2in t + iz}$. Hence as the support of $g'$ is included in $[-M_k, M_k]$,
$$
|K_t(z)|\leq 2M_k \|g'\|_{L^\infty}\; .
$$
By definition of $g$, we have 
$$
g'(n) = \frac{\eta_k''(n)}{(2in t + iz)^2} + \frac{6it \eta_k'}{(2in t + iz)^3} + \frac{3(2it)^2}{(2in t + iz)^4}
$$
therefore
$$
|g'(n)| \lesssim \frac{\|\eta_k''\|_{L^\infty}}{|z|^2} + \frac{\|\eta_k'\|_{L^\infty}|t|}{|z|^3} + \frac{\|\eta_k\|_{L^\infty}t^2}{z^4}\; .
$$
By definition $\eta_k' = \frac1{M_k} \eta'$ and $\frac{|t|}{|z|} \lesssim \frac1{M_k}$, which yields
$$
\|g'\|_{L^\infty} \leq C_\eta \frac{1}{M_k^2 |z|^2}\; .
$$
We get
$$
|K_t(z)|\leq C_\eta  \frac{1}{M_k |z|^2}\; .
$$

By definition, we have 
$$
1_{|x|\leq R} L(t)f(x) = 1_{|x|\leq R} K_t * f (x)
$$
that we divide in two parts
$$
1_{|x|\leq R} L(t)f(x) =I + II
$$
with 
$$
I = 1_{|x|\leq R} \int K_t(y) 1_{|x-y|\leq R+ 3M_k T} f(x-y)dy
$$
and 
$$
II = 1_{|x|\leq R} \int K_t(y) 1_{|x-y|\geq R+ 3M_k T} f(x-y)dy \; .
$$
We have 
$$
\|I\|_{L^p } = \|1_{|x|\leq R}L(t) \Pi_k 1_{|y|\leq R+3M_kT} f(y)\|_{L^p}
$$
which is the first part of the estimate.

For $II$, we use that if $|x|\leq R$ and $|x-y|\geq R + 3M_k T$, then $|y|\geq 3M_k T$. Hence, 
$$
II = 1_{|x|\leq R} \int 1_{|y|\geq 3M_k T}   K_t(y) 1_{|x-y|\geq R+ 3M_k T} f(x-y)dy \; .
$$
By taking its $L^p$ norm, we get
$$
\|II\|_{L^p} \leq \|1_{|y|\geq 3M_k T}   K_t(y)\|_{L^1} \sup_y \|1_{|x-y|\geq R+ 3M_k T}f(x-y)\|_{L^p(x\in [-R,R])}\; .
$$
Using the estimates on $K_t(z)$ we finally get
$$
\|II\| \lesssim \frac{1}{M_k^2 T}  \sup_y \|f\|_{L^p([y-R,y+R])}\; .
$$
\end{proof}

\bibliographystyle{amsplain}
\bibliography{nlsnoncomp} 
\nocite{*}

\end{document}